\documentclass[10pt]{amsart}
\usepackage{latexsym,amsmath,amssymb,graphicx}
\usepackage[all,web]{xy}

\def\urlfont{\DeclareFontFamily{OT1}{cmtt}{\hyphenchar\font='057}
              \normalfont\ttfamily \hyphenpenalty=10000}

\DeclareFontFamily{OT1}{rsfs10}{}
\DeclareFontShape{OT1}{rsfs10}{m}{n}{ <-> rsfs10 }{}
\DeclareMathAlphabet{\mathscript}{OT1}{rsfs10}{m}{n}

\DeclareMathOperator{\im}{Im}       
\DeclareMathOperator{\id}{id}       
\DeclareMathOperator{\Pic}{Pic}     
\DeclareMathOperator{\rk}{rk}       
\DeclareMathOperator{\Sing}{Sing}   
\DeclareMathOperator{\cone}{Cn} 

\title[Homological type of geometric transitions]{Homological type of geometric transitions}

\author[Michele Rossi]{Michele Rossi}

\address{Dipartimento di Matematica, Universit\`a di Torino,
via Carlo Alberto 10, 10123 Torino} \email{michele.rossi@unito.it}

\thanks{This work has been developed despite the effects of the Italian law 133/08 (http://groups.google.it/group/scienceaction ).
        This law drastically reduces public funds to public Italian universities, which is particularly dangerous for free scientific research,
        and it will prevent young researchers from getting a position, either temporary or tenured, in Italy.
        The author is protesting against this law to obtain its repeal.}

\def\P{{\mathbb{P}}}

\def\p2{\mathbb{P}^2}
\def\p3{\mathbb{P}^3}
\def\p4{\mathbb{P}^4}

\def\co{\mathcal{O}}

\def\su{\operatorname{SU}}

\def\rk{\operatorname{rk}}

\def\Exc{\operatorname{Exc}}

\def\Z{\mathbb{Z}}

\def\C{\mathbb{C}}
\def\R{\mathbb{R}}

\def\Q{\mathbb{Q}}
\def\F{\mathbb{F}}

\theoremstyle{plain}
\newtheorem{theorem}{Theorem}[section]
\newtheorem{proposition}[theorem]{Proposition}
\newtheorem{corollary}[theorem]{Corollary}

\newtheorem{lemma}[theorem]{Lemma}

\theoremstyle{remark}
\newtheorem{remark}[theorem]{Remark}

\newtheorem{example}[theorem]{Example}
\newtheorem{examples}[theorem]{Examples}

\theoremstyle{definition}
\newtheorem{definition}[theorem]{Definition}
\newtheorem{defn}{Definition}

\newtheorem*{step I}{Step I}
\newtheorem*{step II}{Step II}
\newtheorem*{step III}{Step III}
\newtheorem*{step IV}{Step IV}
\newtheorem*{conclusion}{Conclusion}
\newtheorem*{acknowledgments}{Acknowledgments}

\newcommand{\halfline}{\vskip6pt}
\newcommand{\cy}{Ca\-la\-bi--Yau }
\newcommand{\ka}{K\"{a}hler }

\begin{document}

\begin{abstract}
The present paper gives an account and quantifies the change in topology induced by small and type II geometric transitions, by introducing the notion of the \emph{homological type} of a geometric transition. The obtained results agree with, and go further than, most results and estimates, given to date by several authors, both in mathematical and physical literature.
\end{abstract}

\maketitle

\tableofcontents

\section*{Introduction}

A \emph{geometric transition} (g.t.) between two \cy threefolds $Y$ and $\widetilde{Y}$
(see Definition \ref{cy-def}) is
the process $T(Y,\overline{Y},\widetilde{Y})$ obtained by ``composing" a birational contraction
$\phi:Y\rightarrow \overline{Y}$, to a
normal threefold $\overline{Y}$, with a complex smoothing (see Definition \ref{g.t.def}). A g.t. is a useful tool for connecting to each other \emph{topologically distinct} \cy
threefolds. This feature is probably the main source of interest
in the study of g.t.'s both in mathematics and physics. \halfline

In mathematics,
the story goes back to deep speculations due to H.~Clemens
\cite{Clemens83}, R.~Friedman \cite{Friedman86}, F.~Hirzebruch
\cite{Hirzebruch87}, J.~Werner \cite{Werner87} and M.~Reid
\cite{Reid87}. In fact, the huge multitude of known topologically
distinct \cy threefolds makes any concept of \emph{moduli space}
for them immediately \emph{wildly reducible}. This unpleasant fact
dramatically clashes with the well known irreducibility of moduli
spaces of both elliptic curves and K3 surfaces, which are the
lower dimensional analogues of \cy threefolds. Actually, Reid
(op. cit.) underlines that, at the beginning (in the forties), the
moduli space of (algebraic) K3 surfaces seemed to F.~Enriques to
be a 19--dimensional variety admitting a countable number of
irreducible components, $\mathcal{M}_g$, one for each value of the
sectional genus, $g\geq 3$ \cite{Enriques46}. Twenty years later
K.~Kodaira \cite{Kodaira64} was able to recover a 20--dimensional
irreducible moduli space, $\mathcal{M}$, for K3 surfaces by leaving
the algebraic category to work in the larger category of analytic,
compact, K\"{a}hler varieties, and discovering that the
\emph{generic} K3 surface is an analytic non--algebraic complex
variety: in particular the moduli space
$\mathcal{M}^{alg}=\bigcup_g \mathcal{M}_g$ of algebraic K3's
embeds in $\mathcal{M}$ as a dense closed subset. The so--called
\emph{Reid's fantasy} suggests that an analogous situation could
happen for \cy threefolds where birational transformations and
 g.t.'s could be
the right instruments to reduce the parameterization of birational
classes of \cy threefolds to an irreducible moduli space of
complex structures over suitable connected sums of complex
hypertori. \halfline

In physics, \cy threefolds made their appearance in the eighties
\cite{CHSW}, as the space spanned by the so--called \emph{internal
degrees of freedom} of a certain $(1+1)$--dimensional world--sheet
field theory describing superstrings' propagation in
$(3+1)$--dimensional Minkowski space--time. The main observables
of the superstring model are then typically determined by this
\emph{internal structure} (the \cy \emph{vacuum}). Therefore, in
spite of the almost uniqueness, via dualities (and, later,
$\mathcal{M}$--theory), of consistent superstring theories, the
superstring model remains undetermined due to the unavoidable
uncertainty on the topology of the \cy vacuum: this is the so
called \emph{vacuum degeneracy problem}. A solution to this
problem was firstly proposed by P.S.~Green and T.~H\"{u}bsch
\cite{Green-Hubsch88let}, \cite{Green-Hubsch88}, who conjectured,
motivated by the contemporary Reid's fantasy, that topologically
distinct \cy vacua could be connected to each other \emph{by means of
conifold transitions} (see Definition \ref{c.t.def}) which should induce a \emph{phase
transition} between corresponding superstring models. This latter
fact, which is the physical counterpart of the mathematical
process given by a conifold transition, was actually understood by
A.~Strominger as a \emph{condensation of massive black holes to
massless ones} \cite{Strominger95}.\halfline

In this context it seems, then, of primary importance to understand
and quantify the change in topology induced by a geometric transition. This is what H. Clemens did for conifold transitions some time ago in \cite{Clemens83}. For more general g.t's, a lot of computations have been carried out by many authors, in the last thirty years, but I wasn't able to find in the literature a clear and/or complete statement about these.

\noindent Moreover, many results are obtained by invoking arguments which are not strictly topological, like, e.g., complex moduli and the Bogomolov--Tian--Todorov Theorem.

In the present paper I will try to organize this problem by introducing the concept of the \emph{homological type of a geometric transition}. Precisely,

\begin{defn}\label{ht} A g.t. $T=T(Y,\overline{Y},\widetilde{Y})$ is said to admit \emph{homological type}
\[
    h[T] = (k',k'',c',c'')
\]
where $k',k''$ are non--negative integers  and $c',c''\in \Z$ with $c'\equiv c''\mod 2$, if $T$ induces the following change on
\begin{itemize}
\item[(a)] \emph{Betti numbers}: \ $b_i(Y)= b_i(\overline{Y}) = b_i(\widetilde{Y})$ for $i\neq
    2,3,4$ and
    \begin{equation*}
    \begin{array}{ccccc}
      b_2(Y) & = & b_2(\overline{Y})+k'+k'' & = & b_2(\widetilde{Y})+k'+k''\ ,\\
      b_3(Y) & = & b_3(\overline{Y})-c' & = & b_3(\widetilde{Y})-c'-c''\ ,\\
      b_4(Y) & = & b_4(\overline{Y})+k'' & = & b_4(\widetilde{Y})+k'+k''\ . \\
    \end{array}
    \end{equation*}
\end{itemize}
Then in particular $T$ induces the following changes on
\begin{itemize}
\item[(b)] \emph{Hodge numbers}:
    \begin{eqnarray*}
      h^{1,1}(Y) &=& h^{1,1}(\widetilde{Y})+k \\
      h^{2,1}(Y) &=& h^{2,1}(\widetilde{Y})-c
    \end{eqnarray*}
    where $k:=k'+k''$ and $c:=(c'+c'')/2$\ ,
\item[(c)] \emph{Euler characteristics}:
    \begin{equation*}
    \left. \begin{array}{c}
      \chi(Y)-\chi(\overline{Y})= k+c'+k'' \\
      \chi(\overline{Y})-\chi(\widetilde{Y})= k'+ c''
    \end{array}\right\}
    \ \Rightarrow\ \chi(Y)-\chi(\widetilde{Y})= 2(k + c)\ .
\end{equation*}
\end{itemize}
Moreover, it turns out that \emph{$k'=0$ if and only if $\overline{Y}$ is $\Q$--factorial} (see Remark \ref{difetto2} and Lemma \ref{difetto=0}).
\end{defn}
\noindent This notation is motivated by equations (b) on Hodge numbers allowing us to conclude that \emph{a g.t. $T(Y,\overline{Y},\widetilde{Y})$ admitting homological type $h(T)=(k',k'',c',c'')$ increases complex moduli \emph{(in physics: \emph{hypermultiplets})} by $c$ and decreases K\"{a}lher moduli \emph{(in physics: \emph{vector multiplets})} by $k$, when passing from the \cy 3--fold $Y$ to the \cy 3--fold $\widetilde{Y}$}.

\medskip\noindent \textbf{Main Results.} Let us first  of all observe that \emph{a general g.t. $T(Y,\overline{Y},\widetilde{Y})$ does not admit a homological type} since the homology change induced by $T$ cannot be  summarized by a string of 4 integers: in fact in general $b_2(\overline{Y})\neq b_2(\widetilde{Y})$ requires the introduction of at least one further integer gauging this last discrepancy. By the way, quantifying the topological change given by a general g.t. is actually a hopeless problem since the exceptional locus $E:=\Exc(\phi)$ of the associated birational contraction $\phi:Y\rightarrow\overline{Y}$ may be very wild! On the other, hand what is observed in the present paper is that a homological type as in Definition \ref{ht} may suffice to describe what happens for some classes of (understandable) g.t's. Precisely, it is proven that:

\medskip\noindent\textbf{Theorem \ref{cambio omologico small}}\ \emph{A small g.t. $T(Y,\overline{Y},\widetilde{Y})$} (see Definition \ref{small g.t. def}) \emph{admits homological type}
\[
    h[T]=(k,0,c',c'')
\]
\emph{where $k$ is the number of homologically independent
    exceptional rational curves composing $E=\Exc(\phi)\subset Y$, $c'$ is the number of independent relations
    linking the homology classes of exceptional rational curves in $Y$ and $c''$ is the number of homologically independent vanishing cycles in $\widetilde{Y}$. Moreover,
    \begin{itemize}
    \item[(i)] the total number of irreducible components of $E$ is
    \[
    n:=\sum_{p\in \Sing(\overline{Y})} n_p = k + c'\ ,
    \]
    where $n_p$ is the number of irreducible (rational) components
    of $E_p:=\phi^{-1}(p)$, for any $p\in
    \Sing(\overline{Y})$,
    \item[(ii)] the \emph{global Milnor number of $\overline{Y}$} is
    \[
    m:=\sum_{p\in \Sing(\overline{Y})} m_p = k + c''\ ,
    \]
    where $m_p$ is the Milnor number of the singular point $p\in \Sing(\overline{Y})$. Hence, $k$ turns out to be also the maximal number of independent relations
    linking the homology classes of vanishing cycles in $\widetilde{Y}$.
    \end{itemize}
    In particular, $T$ is a type I g.t.} (see Definition \ref{type-definizione}) \emph{if and only if $k'=1=k$ and if $T$ is a conifold t., then $c'=c''=c$}.

\medskip\noindent\textbf{Theorem \ref{cambio omologico tipo II}}\ \emph{A type II g.t.} $T(Y,\overline{Y},\widetilde{Y})$ (see Definition \ref{type-definizione}) \emph{admits homological type}
\[
    h[T]=(0,1,c',c'')
\]
\emph{given by}
\begin{itemize}
    \item[(i)] $c'=-1 +b_2(E)-b_3(E)$\ ,
    \item[(ii)] $c''=1-\chi(\widetilde{B})=m_p-b_2(\widetilde{B})$\ ,
\end{itemize}
\emph{where $\widetilde{B}$ is the Milnor fiber of the smoothing $\widetilde{Y}$ and $m_p:=b_3(\widetilde{B})$ is the Milnor number of the unique singular point $p=\phi(E)\in\overline{Y}$. In particular $b_1(\widetilde{B})=0$\ }.

\medskip \noindent Let us say a few words about the previous results.

\noindent When $T$ is a conifold t., Theorem \ref{cambio omologico small} gives precisely the Clemens' results in \cite{Clemens83}, here reported in Theorem \ref{cambio omologico}. Moreover,  relations (ii) and part of relations  in Definition \ref{ht}.(a), with $k'$ and $c''$ as given in the statement of Theorem \ref{cambio omologico small}, were already proved  by Y. Namikawa and J.H.M. Steenbrink in \cite{Namikawa-Steenbrink95} Theorem (3.2) and Example (3.8): as far as I know, these were the most complete known evaluations of homology change induced by a small g.t. until now.

\noindent Theorem \ref{cambio omologico tipo II} can be improved by specializing the assumptions. In fact, it is well known that the exceptional locus $E=\Exc(\phi)$ of the birational contraction $Y\stackrel{\phi}{\rightarrow} \overline{Y}$, associated with a type II g.t. $T(Y,\overline{Y},\widetilde{Y})$, is an irreducible \emph{generalized} del Pezzo surface (\cite{Reid80}, Proposition (2.13)). This means that one of the following happens
\begin{itemize}
\item \emph{$E$ is normal and rational} then Theorem \ref{cambio omologico tipo II} can be improved as follows (see Theorem \ref{cambio omologico II +}):
    \begin{itemize}
    \item[(i)] $c'=9-d-n_E$\ , \emph{where $d=\deg(E):=\omega_E^2$ and $n_E$ is the number of rational exceptional curves composing the exceptional locus of a minimal resolution $\widehat{E}\stackrel{\pi}{\rightarrow}E$\ ,}
    \item[(ii)] \emph{if $d\leq 4$ then $c''=m_p$\ .}
    \end{itemize}
\item \emph{$E$ admits an elliptic singular point} (in the sense of \cite{Reid80}, Definition
    (2.4)) then Theorem \ref{cambio omologico tipo II} can be improved as follows (see Theorem \ref{cambio omologico tipo II-ellittico}):
    \begin{itemize}
    \item[(i)] \emph{$b_2(E)=1$ and $b_3(E)=2$, giving $c'=-2$\ ,}
    \item[(ii)] \emph{$c''=m_p$ and in particular it must be even.}
    \end{itemize}
\item \emph{$E$ is a non--normal del Pezzo surface} (in the sense of \cite{Reid94}) then Theorem \ref{cambio omologico tipo II} can be improved as follows (see Theorem \ref{cambio omologico tipo II-nonnormale}):
    \begin{itemize}
    \item[(i)] \emph{$b_2(E)\in\{1,2\}$ and $b_3(E)=0$, giving $c'\in\{0,1\}$, respectively},
    \item[(ii)] $\chi(\widetilde{B})\equiv b_2(E)\mod 2 $\ .
    \end{itemize}
\end{itemize}
 The first equation in Definition \ref{ht}.(b) in the case of a type II g.t., hence giving $k=1$, was already proved in \cite{KK} Proposition 3.1. In particular, if \emph{$E=\Exc(\phi)$ is smooth} then a comparison between the equations in Definition \ref{ht}.(b) and Theorem 3.3 and Remark 3.6 in \cite{KK}, gives conditions on the Euler characteristic $\chi(\widetilde{B})$ and the Milnor number $m_p$ and then the list in Remark \ref{Milnor-minorazioni}. In particular, this shows that we cannot expect the Milnor fiber near $p\in\overline{Y}$ to be a bouquet of 3--dimensional spheres when $d:=\deg(E)= 6$ or 7, the contrary of what happens for $d\leq 4$. Moreover, a further comparison with \cite{Morrison-Vafa96} \S 7.1 and \cite{Morrison-Seiberg97} \S 3 gives an interesting interpretation of the homological type of a type II g.t. in terms of (dual) Coxeter numbers of suitable Weyl groups, as in (\ref{coxeter}).

\medskip Employed techniques are almost completely topological: construction of strong deformation retractions, Mayer--Vietoris  and relative homology long exact sequences, dualities in homology, Leray spectral sequences and associated lower terms exact sequences. Transcendental methods (exponential sequence on a \cy 3--fold and its pushed forward by birational contraction on a normal and singular 3--fold) are used simply to focus on how the Picard group is actually a topological invariant of all the 3--folds involved in a g.t.\ .

\medskip The present paper is organized as follows. \S 1 reviews what a g.t. is, the related nomenclature, and Wilson's classification of birational contractions of \cy 3--folds. In \S 2 Milnor's and Looijenga's local analysis of isolated singularities are reviewed and the strong deformation retractions useful in the following are introduced (Propositions \ref{contrazione i.s.} and \ref{omotopia scoppiata}). In \S 3 we review a result proved in \cite{R}, rewritten in terms of the homological type of conifold transitions. Finally \S 4 and \S 5 state and prove the main results, the former for small g.t's and the latter for type II g.t's\ .

\begin{acknowledgments} I would like to thank Alberto Collino for useful suggestions.
\end{acknowledgments}

\section{Geometric transitions}

\begin{definition}[\cy 3--folds]\label{cy-def}
A smooth, complex, projective 3--fold $Y$ is called \emph{\cy} if
\begin{enumerate}
    \item[(a)] $\mathcal{K}_Y \cong \mathcal{O}_Y$\ ,
    \item[(b)] $h^{1,0}(Y) = h^{2,0}(Y) = 0$\ .
\end{enumerate}
\end{definition}

\begin{remark}There are a lot of more or less equiv\-a\-lent def\-i\-ni\-tions
    of \cy 3--folds e.g.: a \ka complex, compact 3--fold admitting either (1) a
    \emph{Ricci flat metric} (Calabi conjecture and Yau theorem), or (2) a flat,
    non--degenerate, holomorphic 3--form, or (3) holonomy group a subgroup
    of $\su (3)$ (see \cite{Joice2000} for a complete description of
    equivalences and implications).

    \noindent In the algebraic context, the given definition gives the 3--dimensional
    analogous of smooth elliptic curves and smooth $K3$ surfaces.
\end{remark}

\begin{examples}
\begin{enumerate}
    \item[(a)] Smooth hypersurfaces of degree $5$ in $\P ^4$\ ,
    \item[(b)] the general element of the anti--canonical system of a
    \emph{sufficiently good} 4--dimensional toric Fano variety (see
    \cite{Batyrev94}),
    \item[(c)] suitable complete intersections.... (iterate the previous
    examples),
    \item[(d)] the double covering of $\P ^3$ ramified along a smooth
    surface of degree 8 in $\P ^3$ (octic double solid).
\end{enumerate}
\end{examples}

\begin{definition}[Geometric transitions]\label{g.t.def}(cfr. \cite{Morrison99},
\cite{Cox-Katz99}, \cite{GR02}, \cite{R}) Let $Y$ be a \cy 3--fold
and $\phi : Y\rightarrow \overline{Y}$ be a \emph{birational
contraction} onto a \emph{normal} variety. If there exists a
complex deformation (\emph{smoothing}) of $\overline{Y}$ to a \cy
3--fold $\widetilde{Y}$, then the process of going from $Y$ to
$\widetilde{Y}$ is called a \emph{geometric transition} (for short
\emph{transition} or \emph{g.t.}) and denoted by
$T(Y,\overline{Y},\widetilde{Y})$ or by the diagram
\[
    \xymatrix@1{Y\ar@/_1pc/ @{.>}[rr]_T\ar[r]^{\phi}&
                \overline{Y}\ar@{<~>}[r]&\widetilde{Y}}\ .
\]
A transition $T(Y,\overline{Y},\widetilde{Y})$ is called
\emph{trivial} if $\widetilde{Y}$ is a deformation of $Y$.
\end{definition}

\begin{definition}[Conifold transitions]\label{c.t.def}
A g.t.
\[
    \xymatrix@1{Y\ar@/_1pc/ @{.>}[rr]\ar[r]^{\phi}&
                \overline{Y}\ar@{<~>}[r]&\widetilde{Y}}
\]
\vskip.1cm \noindent is called \emph{conifold} (\emph{c.g.t.} for
short) and denoted $CT(Y,\overline{Y},\widetilde{Y})$, if
$\overline{Y}$ admits only \emph{ordinary double points} (nodes or
o.d.p.) as singularities.
\end{definition}

\begin{example}[cfr. \cite{GMS95}]\label{l'esempio}
The following is a non--trivial c.g.t.. For details see \cite{R},
1.3.

\noindent Let $\overline{Y}\subset\P ^4$ be the \emph{generic
quintic 3--fold containing the plane $\pi : x_3 = x_4 = 0$}. Its
equation is
\[
    x_3 g(x_0,\ldots ,x_4) + x_4 h(x_0,\ldots ,x_4) = 0
\]
where $g$ and $h$ are generic homogeneous polynomials of degree 4.
$\overline{Y}$ is then singular and
\[
    \Sing (\overline{Y}) = \{ [x]\in \P ^4 | x_3=x_4=g(x)=h(x)=0\}
    = \{ 16 \ \text{nodes}\}\ .
\]
Blow up $\P^4$ along the plane $\pi$ and consider the proper
transform $Y$ of $\overline{Y}$. Then:
\begin{itemize}
    \item $Y$ is a smooth, \cy 3--fold,
    \item the restriction to $Y$ of the blow up morphism gives a
    crepant resolution $\phi: Y\rightarrow\overline{Y}$.
\end{itemize}
The obvious smoothing of $\overline{Y}$ given by the generic
quintic 3--fold $\widetilde{Y}$ completes the c.g.t.
$CT(Y,\overline{Y},\widetilde{Y})$.
\end{example}

\begin{definition}[Primitive contractions and transitions]
A birational contraction from a \cy 3--fold to a normal one is
called \emph{primitive} if it cannot be factored into birational
morphisms of normal varieties. A g.t.
\[
    \xymatrix@1{Y\ar@/_1pc/ @{.>}[rr]_T\ar[r]^{\phi}&
                \overline{Y}\ar@{<~>}[r]&\widetilde{Y}}
\]
is called \emph{primitive} if the associated birational
contraction $\phi$ is primitive.
\end{definition}

\begin{proposition}
Let $T(Y,\overline{Y},\widetilde{Y})$ be a g.t. and $\phi
:Y\rightarrow\overline{Y}$ the associated birational contraction.
Then $\phi$ can always be factored into a composite of a
\emph{finite} number of primitive contractions.
\end{proposition}

\begin{proof}
The statement follows from the fact that any primitive contraction
reduces by 1 the Picard number $\rho = \rk (\Pic(Y)) =
h^{1,1}(Y)$.
\end{proof}

\begin{theorem}[Classification of primitive contraction
\cite{Wilson92}, \cite{Wilson93}]\label{classificazione} Let $\phi
: Y\rightarrow\overline{Y}$ be a primitive contraction from a \cy
threefold to a normal variety and let $E$ be the exceptional locus
of $\phi$. Then one of the following is true:
\begin{description}
    \item[type I] $\phi$ is \emph{small} which means that $E$ is
    composed of finitely many rational curves;
    \item[type II] $\phi$ contracts a divisor down to a point; in
    this case $E$ is irreducible and in particular it is a
    (generalized) \emph{del Pezzo surface} (see \cite{Reid80})
    \item[type III] $\phi$ contracts a divisor down to a curve
    $C$; in this case $E$ is still irreducible and it is a conic
    bundle over a smooth curve $C$.
\end{description}
\end{theorem}

\begin{definition}[Classification of primitive
transitions]\label{type-definizione} A transition
$T(Y,\overline{Y},\widetilde{Y})$ is called \emph{of type I, II or
III} if it is \emph{primitive} and if the associated birational
contraction $\phi : Y\rightarrow\overline{Y}$ is of type I, II or
III, respectively.
\end{definition}

\begin{definition}[Small geometric transition]\label{small g.t. def} A transition
\[
    \xymatrix@1{Y\ar@/_1pc/ @{.>}[rr]_T\ar[r]^{\phi}&
                \overline{Y}\ar@{<~>}[r]&\widetilde{Y}}
\]
will be called \emph{small} if $\phi$ is the composition of
primitive contractions of type I.
\end{definition}

\section{Local analysis of an isolated singularity}\label{sing-locale}

Let $p\in\overline{Y}$ be a \emph{$n$--dimensional isolated singularity} i.e. there
exists a $n$--dimensional local analytic neighborhood $\overline{U}$ of $p$ bi-holomorphic to the zero locus of a polynomial
map
\[
\mathbf{f}=(f_1,\ldots,f_M):\C^N\rightarrow \C^M
\]
where $N-M\leq n$ and such that $\mathbf{f}$ has an isolated critical point in $0\in\C^N$ and $\mathbf{f}$ is a
submersion over $\C^N\setminus\{ 0\}$. If $N-M=n$ then $p\in \overline{Y}$
is an \emph{isolated complete intersection singularity}
(i.c.i.s.). Moreover, if $M=1$, then $N=n+1$, and $p\in \overline{Y}$ is an \emph{isolated hypersurface singularity} (i.h.s).
Let $D_{\varepsilon}$ denote the closed
$\varepsilon$--ball centered in $0\in\C^N$ whose boundary is the
$(2N-1)$--dimensional $\varepsilon$--sphere $S_{\varepsilon}$.

\begin{definition}[Good representatives and Milnor fibers, \cite{Looijenga} Section 2.B]\label{fibra Milnor}
Let $T^m$ be a $m$--dimensional contractible neighborhood of $0\in
\C^M$, with $m\leq M$, and consider
$\mathcal{U}^m:=\mathbf{f}^{-1}(T^m)$. Then
\[
\mathbf{f}:\mathcal{U}^m\longrightarrow T^m
\]
is called a \emph{good representative of $p$}. We will omit the
dimension $m$ if unnecessary.

\noindent Set $\widetilde{U}:=U_t$ for $t\in T\setminus\{0\}$. Then,
for small $\varepsilon>0$, the intersection
\[
\widetilde{B}=\widetilde{U}\cap D_\varepsilon
\]
is called \emph{the Milnor fiber of $p$}.
\end{definition}

\begin{definition} Let $X\subset\C^N$ be a subset. The following subset of $\C^N$
\[
\cone _0^{\varepsilon} (X) := \left\{ tx\ |\ \forall t\in [0,\varepsilon]\subset\R\ ,\
\forall x\in X \right\}
\]
will be called the $\varepsilon$--\emph{cone projecting $X$}. The 1--cone will be simply called the \emph{cone} projecting $X$ and denoted by $\cone _0(X)$. Moreover, the $\infty$--cone $\cone _0^{\infty} (X)$ will be also called the
\emph{unbounded cone} projecting $X$.
\end{definition}

\begin{theorem}[Local topology of an isolated singularity,
\cite{Looijenga} Proposition (2.4), \cite{Milnor68} Theorem 2.10, Theorem
5.2]\label{i.s.topologia} Since $p\in\overline{U}$ is the unique critical point of $\mathbf{f}$, for any
$\varepsilon>0$ there exists a homeomorphism $\psi_{\varepsilon}$ between the intersection
\[
\overline{B}:=\overline{U}\cap D_{\varepsilon}
\]
and the $\varepsilon$--cone projecting $K:=\overline{U}\cap
S_{\varepsilon}$. Moreover, ${1\over \varepsilon}\psi_{\varepsilon}$ gives a homeomorphism $\overline{B}\cong\cone _0(K)$.
\end{theorem}

\noindent $K$ is called \emph{the link of $p$} : it only depends
on the abstract analytic germ $(\overline{Y},p)$ (see
\cite{Looijenga}, Corollary (2.6)).

\begin{proposition}\label{contrazione i.s.}
In the same notation as before, $\overline{B}$ and $\widetilde{B}$
are strong deformation retract of $\overline{U}$ and
$\widetilde{U}$, respectively.
\end{proposition}

\begin{proof} Let us first of all observe that, by direct limit construction, there exists an homoeomorphism
\begin{equation}\label{cono illimitato}
    \varphi : \xymatrix@1{\overline{U} \ar[r]^-{\cong} & \cone _0^{\infty}(K) }
\end{equation}
where $K:=\overline{U}\cap S_{\varepsilon}$ for a sufficiently small $\varepsilon>0$. In fact, for any $0<\varepsilon'<\varepsilon''$, Theorem \ref{i.s.topologia} gives the following commutative diagram
\[
\xymatrix{\overline{U}\cap D_{\varepsilon''}\ar[r]^{\psi_{\varepsilon''}}_{\cong}& \cone _0^{\varepsilon''}(K)\\
          \overline{U}\cap D_{\varepsilon'}\ar[r]^{\psi_{\varepsilon'}}_{\cong}\ar@{^{(}->}[u]^{{\varepsilon''\over\varepsilon'}}& \cone _0^{\varepsilon'}(K)\ar@{^{(}->}[u]^{{\varepsilon''\over\varepsilon'}}  }
\]
which, passing to direct limits $\overline{U}$ and $\cone _0^{\infty} (K)$, gives the claimed homeomorphism $\varphi$. Moreover, by setting $\psi:={1\over\varepsilon}\psi_{\varepsilon}$, we get a homeomorphism
\begin{equation}\label{cono compatto}
    \psi : \xymatrix@1{\overline{B} \ar[r]^-{\cong} & \cone _0(K) }\ .
\end{equation}
Let us start by showing the strong retraction
$\overline{U}\approx\overline{B}$: calling $i:\overline{B}\hookrightarrow\overline{U}$ the
natural inclusion, we will write explicitly the
retraction $\overline{r}:\overline{U}\rightarrow\overline{B}$ and
the homotopy
$\overline{F}:\overline{U}\times[0,1]\rightarrow\overline{U}$ such
that $\id_{\overline{U}}\sim_{\overline{F}} i\circ\overline{r}$,
since they will be useful in the following. Let
\[
r:\xymatrix@1{\cone _0^{\infty}(K)\ar[r] & \cone _0(K)}
\]
the obvious retraction defined by setting
$r(x):=\left\{\begin{array}{cc}
  x & \text{if $x\in\cone _0(K)$} \\
  x/|x| & \text{otherwise} \\
\end{array}\right.$.

\noindent Consider the straight line homotopy
\[
H:\xymatrix@1{\cone _0^{\infty}(K)\times [0,1]\ar[r] & \cone
_0^{\infty}(K)}
\]
defined by setting $H(x,t):=\left\{\begin{array}{cc}
  x & \text{if $x\in\cone _0(K)$} \\
  (1-t)x+tx/|x| & \text{otherwise} \\
\end{array}\right.$.

\noindent $H$ is clearly continuous by the gluing lemma over $K$
whose points have unitary modulus. Moreover:\begin{itemize}
    \item $H(x,0) = \id_{\cone _0^{\infty}(K)}(x) = x$ and $H(x,1) = r
    (x)$, which means that $\cone _0(K)$ is a \emph{deformation
    retract} of $\cone _0^{\infty}(K)$,
    \item $H(x,t) = x$ for every $x\in\cone _0(K)$, which means
    that $\cone _0(K)$ is a \emph{strong} deformation
    retract of $\cone _0^{\infty}(K)$.
\end{itemize}
We are now able to construct the retraction $\overline{r}$ and the
homotopy $\overline{F}$ as follows
\begin{eqnarray}\label{omotopia singolare}
\nonumber
  \overline{r} &:=& \psi^{-1}\circ r\circ\varphi \\
  \overline{F}(z,t) &:=& \varphi^{-1}\left(H(\varphi(z),t)\right)\quad
  ,\quad \forall(z,t)\in\overline{U}\times [0,1]
\end{eqnarray}
(notice that $\psi = \varphi\circ i \Rightarrow
i\circ\psi^{-1}=\varphi^{-1}|_{\cone_0(K)}$).

\noindent To prove the strong retraction
$\widetilde{U}\approx\widetilde{B}$ set $\mathcal{B}_{\varepsilon}:=
\mathcal{U}_{\varepsilon}\cap D_{\varepsilon}$, $\mathcal{K}_{\varepsilon}:= \mathcal{U}_{\varepsilon}\cap S_{\varepsilon}$ and
consider the restricted fibration
\[
f_{\llcorner}:\xymatrix@1{\mathcal{U}_{\varepsilon}\setminus\mathcal{B}_{\varepsilon}\cup\mathcal{K}_{\varepsilon}\ar[r]
& T}\ .
\]
Up to shrink $T$, the map $f_{\llcorner}$ gives a fibration in
$2n$--dimensional $\mathcal{C}^{\infty}$--manifolds with boundary
whose fibers are diffeomorphic to each other by the Ehresmann
fibration theorem. Let then
\[
\eta
:\xymatrix@1{\widetilde{U}\setminus\widetilde{B}\cup\widetilde{K}\ar[r]^-{\cong}
& \overline{U}\setminus\overline{B}\cup\overline{K}}
\]
be the induced diffeomorphism between the fibers
$f_{\llcorner}^{-1}(t)\cong f_{\llcorner}^{-1}(0)$. We can then
construct a retraction $\widetilde{r}:
\widetilde{U}\rightarrow\widetilde{B}$ and a homotopy
$\widetilde{F}:\widetilde{U}\times [0,1]\rightarrow\widetilde{U}$
as follows
\begin{eqnarray*}
  \widetilde{r}(z) &:=& \left\{\begin{array}{cc}
    z & \text{if $z\in\widetilde{B}$} \\
    \eta^{-1}\circ \overline{r}\circ \eta (z) & \text{otherwise} \\
  \end{array}\right.\\
  \widetilde{F}(z,t) &:=& \left\{\begin{array}{cc}
    z & \text{if $z\in\widetilde{B}$} \\
    \eta^{-1}\left(\overline{F}(\eta (z), t)\right) & \text{otherwise} \\
  \end{array}\right.
\end{eqnarray*}
Both $\widetilde{r}$ and $\widetilde{F}$ are continuous by the
gluing lemma over $\widetilde{K}$ and the proof concludes
immediately  by verifying that
$\id_{\widetilde{U}}\sim_{\widetilde{F}} \widetilde{i}\circ
\widetilde{r}$, where
$\widetilde{i}:\widetilde{B}\hookrightarrow\widetilde{U}$ is the
inclusion.
\end{proof}

\begin{theorem}[Local homotopy type of the smoothing, \cite{Looijenga} (5.6), (5.8) and Corollary (5.10), \cite{Milnor68}
Theorems 5.11, 6.5, 7.2]\label{i.s.omotopia}
The Milnor fiber $\widetilde{B}$ is homotopy equivalent to a finite cell complex of dimension $\leq n$. Moreover, if $0\in\mathbf{f}^{-1}(0)$ is an i.c.i.s. then $\widetilde{B}$ is $(n-1)$-connected and has the homotopy type of a finite bouquet of $n$--spheres.
\end{theorem}

In the following the $n$-th Betti number $b_n(\widetilde{B})$ of the Milnor fiber will be called \emph{the Milnor number of $p\in\overline{U}\subseteq\overline{Y}$ and denoted by $m_p$.}

\begin{remark} If $p\in \overline{U}$ is a $n$--dimensional i.h.s. then the Milnor number $m_p$ coincides
with the \emph{multiplicity} of $p$ as a critical point of the polynomial map $f$,
which is the multiplicity of $0\in\C^{n+1}$ as solution to
the collection of polynomial equations
\[
\frac{\partial f}{\partial x_1}=\cdots=\frac{\partial f}{\partial x_{n+1}}=0\
.
\]
The Milnor number is then given by
\[
m_p=\dim _{\C} \left(\C\{x_1,\ldots ,x_{n+1}\}\left/\left(\frac{\partial f}{\partial
x_1},\cdots ,\frac{\partial f}{\partial x_{n+1}} \right)\right)\right.
\]
where $\C\{x_1,\ldots ,x_{n+1}\}$ is the $\C$--algebra of converging power series.
\end{remark}

\begin{proposition}[Local topology of the resolution]\label{omotopia scoppiata} Let $\phi: U\rightarrow\overline{U}$ be a birational resolution of
the i.s. $p\in\overline{U}$ and set $E:=\Exc(\phi)$. The pull back by $\phi$ of the strong deformation
retraction $\overline{U}\approx\overline{B}\approx \{p\}$ gives rise
to a strong deformation retraction $U\approx B\approx E$, where $B:=\phi^{-1}(\overline{B})$.
\end{proposition}

\begin{proof}[Proof of Lemma \ref{omotopia scoppiata}] Consider the homeomorphisms $\varphi,\psi$ given by (\ref{cono illimitato}) and (\ref{cono compatto}), respectively. Then the straight line homotopy giving the contraction
$\overline{B}\approx \{p\}$ is the following
\begin{eqnarray*}
  \overline{G} :&&\xymatrix@1{\overline{B}\times [0,1]\ar[rr] && \overline{B}}  \\
   &&\xymatrix@1{\quad\quad (z,t)\quad\ar@{|->}[r]& \quad\psi^{-1}((1-t)\psi(z))}
\end{eqnarray*}
Therefore the contraction $\overline{U}\approx \{p\}$ is
realized by the composition $\overline{G}\ast\overline{F}$
where $\overline{F}$ is the homotopy defined in (\ref{omotopia
singolare}). Pulling back by the resolution
$\phi:U\rightarrow\overline{U}$ we get retractions and homotopies
\begin{eqnarray*}
  r &:& \xymatrix@1{U\ar[r] & E} \\
  G\ast F &:& \xymatrix@1{U\times [0,1]\ar[r] & U}
\end{eqnarray*}
defined by setting
\begin{eqnarray*}
  r(z) &:=& \left\{\begin{array}{cc}
    z & \text{if $z\in E$} \\
    \widehat{l}_{\phi(z)}\cap E & \text{otherwise} \\
  \end{array} \right.\\
  G\ast F (z,t) &:=& \left\{\begin{array}{cc}
    z & \text{if $z\in E$} \\
    \phi^{-1}\left(\overline{F}(\phi(z),2t)\right) & \text{if $(z,t)\in
    (U\setminus E)\times [0,{1\over 2}]$} \\
    \phi^{-1}\left(\overline{G}\left(\overline{F}(\phi(z),1),2t-1\right)\right)
    &\text{if $(z,t)\in(U\setminus E)\times [{1\over 2},1)$}\\
    \widehat{l}_{\phi(z)}\cap E & \text{if $(z,t)\in(U\setminus E)\times \{1\}$}
  \end{array} \right.
\end{eqnarray*}
where $\widehat{l}_{\phi(z)}$ is the \emph{strict transform by
$\phi$} of the segment
$$l_{\phi(z)}=\{\varphi^{-1}((1-t)\varphi(\phi(z)))\ |\
t\in[0,1]\}$$
i.e. it is the closure of the subset
$\{\Phi^{-1}((1-t)\Phi(z))\ |\ t\in[0,1)\}\subset U$, where $\Phi:=\varphi\circ\phi$. Then
$\widehat{l}_{\phi(z)}\cap E$ is its unique closure point. This
gives the continuity of $r$ and of $G\ast F$ by the gluing
lemma over $E$ and over $U\times\{{1\over 2},1\}$.
Moreover,\begin{itemize}
    \item $G\ast F (z,0) = z$ and $G\ast F(z,1) =
    r(z)$, for every $z\in U$, which means that $E$ is a
    deformation retract of $U$,
    \item $G\ast F(z,t) = z$ for every $(z,t)\in E\times
    [0,1]$, which means that $E$ is a strong deformation retract of
    $U$.
\end{itemize}
\end{proof}

\section{Homological type of a conifold transition}

Consider the following c.g.t.
\begin{equation}\label{conifold}
\xymatrix@1{Y\ar@/_1pc/ @{.>}[rr]_{CT}\ar[r]^{\phi}&
                \overline{Y}\ar@{<~>}[r]&\widetilde{Y}}
\end{equation}
Then, by definition, $\Sing(\overline{Y})$ is entirely composed by a finite number of nodes. The following theorem can be obtained by summarizing several results of many authors. It gives a complete account of changing in topology induced by a c.g.t.

\begin{theorem}[Global changing in topology for a conifold transition
\cite{Clemens83}, \cite{Reid87}, \cite{Werner-vanGeemen90},
\cite{Tian92}, \cite{Namikawa-Steenbrink95},
\cite{Morrison-Seiberg97}, ...]\label{cambio omologico} Consider
the c.g.t. (\ref{conifold}). Then
\begin{equation}\label{ht_CT}
    h[CT]=(k,0,c,c)
\end{equation}
where
\begin{itemize}
    \item $k$ is the number of homologically independent
    exceptional $\P^1$s in $Y$,
    \item $c$ is the number of homologically independent
    vanishing cycles in $\widetilde{Y}$.
\end{itemize}
Moreover, if $N:=\Sing
    (\overline{Y})$ is the number of nodes in $\overline{Y}$, then $N=k+c$\ .
\end{theorem}

A detailed topological proof of this theorem, based on the local Clemens' analysis
described in \cite{Clemens83}, is given in
\cite{R}, section 3.

\begin{remark}\label{difetto}
Note that point (a) in Definition \ref{ht}, with $k$ and $c$ as in Theorem \ref{cambio omologico}, implies that the
conifold $\overline{Y}$ do not satisfy Poincar\'{e} Duality. The
difference $b_4(\overline{Y}) - b_2(\overline{Y})=k$ is called
\emph{the defect of $\overline{Y}$} \cite{Namikawa-Steenbrink95}.
See also the following remark \ref{difetto2}.

\noindent This fact can be deeply understood by means of \emph{homology of intersections spaces} and \emph{intersection homology}, as recently introduced and explained by M.~Banagl in \cite{Banagl} \S 5, giving a nice account of relations with type II string theories and mirror symmetry in physics (see also the following remark \ref{cpx<->ka}).
\end{remark}

\begin{remark}\label{cpx<->ka} Point (b) in Definition \ref{ht}, with $k$ and $c$ as in Theorem \ref{cambio omologico}, admits the following geometric (and physical)
interpretation: \emph{a c.g.t. increases complex moduli \emph{(in phisics: \emph{hypermultiplets})} by the
number $c$ of homologically independent vanishing cycles  and
decreases \ka moduli \emph{(in phisics: \emph{vector multiplets})} by the number $k$ of homologically
independent exceptional rational curves}.
\end{remark}

\begin{remark}[Example \ref{l'esempio} continued] If Theorem \ref{cambio omologico} is applied to the c.g.t. in Example \ref{l'esempio} one finds that $k'=k=1, k''=0$, since $Y\stackrel{\phi}{\longrightarrow} \overline{Y}$ is primitive, as induced by a blow up. Then $c=N-k=16-1=15=c'=c''$ and recalling that $\widetilde{Y}$ is a generic quintic 3--fold in $\P^4$ one gets the following table of Betti numbers:
\begin{equation*}
\begin{tabular}{lclclclc}
  \hline \hline
    \hskip11.4pt   \vrule     &\quad     &$b_2$ &\quad  & $b_3$&\quad
  &$b_4$\\
  \hline
  \hskip11.4pt\vrule&&&&&&&\\
  $Y$ \vrule   &\quad    & 2 &\quad & 174 &\quad & 2  &\\
  $\overline{Y}$ \vrule            &\quad      &  1 &\quad & 189&\quad  & 2 &\\
  $\widetilde{Y}$ \vrule &\quad  & 1 &\quad & 204 &\quad & 1 &\\
   \hline\hline
\end{tabular}
\end{equation*}
Then $\overline{Y}$ has defect 1, as observed in Remark \ref{difetto}.
\end{remark}

\section{Homological type of a small geometric transition}

Let
\begin{equation}\label{small g.t.}
\xymatrix@1{Y\ar@/_1pc/ @{.>}[rr]_{T}\ar[r]^{\phi}&
                \overline{Y}\ar@{<~>}[r]&\widetilde{Y}}
\end{equation}
be a \emph{small} g.t.. Then $\Sing (\overline{Y})$ is composed of
a finite number of \emph{canonical singularities} (see
\cite{Reid80}, section 1, for the definition) since $\dim
\phi^{-1}(p)=1$, for any singular point $p\in \overline{Y}$.

\noindent Recall that a \emph{compound Du Val} (cDV) singularity
is a 3--fold point $p$ such that, for a hyperplane section $H$
through $p$, $p\in H$ is a Du Val surface singularity i.e. an
A--D--E singular point (see \cite{Reid80}, sections 0 and 2, and
\cite{BPvdV84}, chapter III). The singular locus $P:=\Sing
(\overline{Y})$ and the exceptional locus $E:=\phi^{-1}(P)$ have
then a well known geometry reviewed by the following statement.

\begin{theorem}[\cite{Reid83}, \cite{Laufer81}, \cite{Pinkham81},
\cite{Morrison85}, \cite{Friedman86}]\label{risoluzione} Given a
small g.t. as in (\ref{small g.t.}) then:
\begin{enumerate}
    \item[(a)] any $p\in P$ is a cDV singularity,
    \item[(b)] for any $p\in P$, $E_p:=\phi^{-1}(p)$ is a connected
    union of rational curves meeting transversally, whose
    configuration is dually represented either by one of the following
    graphs
    \begin{eqnarray*}
       && \xymatrix@1{{A_n\ :\quad}&{\bullet}\ar@{-}[r]&{\bullet}\ar@{-}[r]&{\bullet}\ar@{-}[r]&
                        {\bullet}\ar@{.}[r]&{\bullet}&\quad\text{($n\geq 1$ vertices)}}
                             \\
       && \xymatrix{&{\bullet}\ar@{-}[dr]& & & &\\
                    {D_n\ :\quad}&&{\bullet}\ar@{-}[r]
                                &{\bullet}\ar@{.}[r]&{\bullet}&\quad\text{($n\geq 4$ vertices)}\\
                                &{\bullet}\ar@{-}[ur]& & & &}
                                 \\
       && \xymatrix{{E_n\ :}&{\bullet}\ar@{-}[r]&{\bullet}\ar@{-}[r]&{\bullet}\ar@{-}[r]\ar@{-}[d]&
                                 {\bullet}\ar@{-}[r]&
                                 {\bullet}\ar@{.}[r]&{\bullet\quad \text{($n=6,7,8$ vertices).}} \\
                                 &&&{\bullet}&&&}\\
    \end{eqnarray*}
    or, if $p$ is a non--planar singularity, by one the following
    graphs
    \begin{eqnarray*}
      && \xymatrix{&{\bullet}\ar@{-}[dd]\ar@{-}[dr]& & & & \\
                   {\widetilde{D}_n\ :\quad}&&{\bullet}\ar@{-}[r]&{\bullet}\ar@{.}[r]&{\bullet}
                   &\quad\text{($n\geq 3$ vertices)}\\
                               &{\bullet}\ar@{-}[ur]& & & &}\\
      && \xymatrix{&&&{\bullet}\ar@{-}[dl]\ar@{-}[dr]&&&\\
                   {\widetilde{E}_n\ :}&{\bullet}\ar@{-}[r]&{\bullet}\ar@{-}[rr]&&{\bullet}\ar@{-}[r]
                   &{\bullet}\ar@{.}[r]&{\bullet\quad \text{($n=5,6,7$ vertices).}}}\\
    \end{eqnarray*}
where triangles are dual graphs representing the transverse
intersection of three rational curves at a single
point.
\end{enumerate}
\end{theorem}
Clearly every conifold transition is a small g.t. admitting exceptional trees of $A_1$ type. The following example presents a small non--conifold transition.

\begin{example}[A small non--conifold g.t., \cite{Namikawa02} Example 1.11 and Remark 2.8, \cite{R2}]\label{esempio-Nami}
Let $S$ be the rational elliptic surface with
sections obtained as the Weierstrass fibration associated with the
bundles homomorphism
\begin{eqnarray}\label{W-Nami}
    \beta&:&\xymatrix@1{\mathcal{E}=\mathcal{O}_{\P^1}(3)\oplus\mathcal{O}_{\P^1}(2)
\oplus\mathcal{O}_{\P^1}\ \ar[rr]&&\quad
    \mathcal{O}_{\P^1}(6)}\\
\nonumber
    &&\xymatrix@1{\hskip1cm (x,y,z)\hskip1cm\ar@{|->}[rr]&&\quad - x^2z + y^3 + B(\lambda)\
    z^3}
\end{eqnarray}
for a generic $B\in H^0(\P^1,\mathcal{O}_{\P^1}(6))$ i.e. $S$ is
the zero locus of $\beta$ in the projectivized bundle
$\P(\mathcal{E})$. Then:
\begin{itemize}
    \item \emph{the natural fibration $S\rightarrow\P^1$ has generic smooth fiber
and 6 distinct \emph{(since $B$ is generic)} cuspidal fibers},
\item \emph{the fiber product $\overline{Y}:= S\times_{\P^1}S$ is a threefold admitting 6
    distinct singularities of type $II\times II$, in the standard Kodaira
    notation \cite{Kodaira64}, whose local equation in $\C^4$ is given by}
\begin{equation}\label{cuspide}
    x^2-u^2=y^3-v^3\ ,
\end{equation}
\item \emph{$\overline{Y}$ is a special fiber of the family of fiber products $S_1\times_{\P^1} S_2$
    of rational elliptic surfaces with sections \cite{Schoen88}: in particular every rational elliptic surface can be thought as the blow up of $\P^2$ at the base locus of a bi--cubic rational map $\P^2\stackrel{[a:b]}{\dashrightarrow}\P^1$} (\cite{MP} Prop. 6.1) \emph{implying that $\overline{Y}$ can be thought as a special fiber of the family of smooth resolutions of bi--cubic hypersurfaces} (\cite{R2} Remark 2.2)
    \begin{equation}\label{bicubica}
        \hskip8pt\left\{a(x)b'(x')=a'(x')b(x)\right\}\subset\P^2[x]\times\P^2[x']\ ,\quad a,b,a',b'\in H^0(\co_{\P^2}(3))\ ,
    \end{equation}

\item \emph{$\overline{Y}$ admits a \emph{small} resolution $Y\stackrel{\phi}{\longrightarrow} \overline{Y}$
    whose exceptional locus is composed by 6 disjoint couples of rational curves
    intersecting in one point i.e. 6 disjoint $A_2$ exceptional
    trees in the notation of Theorem \ref{risoluzione}}
    (\cite{Namikawa02} \S 0.1, \cite{R2} Proposition 3.1).
\end{itemize}
Let us say a few words about the construction of the resolution $Y\stackrel{\phi}{\longrightarrow}\overline{Y}$. The fibred product $\overline{Y}:= S\times_{\P^1}S$ can be thought as embedded in $\P:=\P(\mathcal{E})\times\P(\mathcal{E})\times\P^1$ as follows
\begin{equation}\label{Nami}
    \P:=\P(\mathcal{E})\times\P(\mathcal{E})\times\P^1[\lambda]\supset \overline{Y}\ :\
    \left\{\begin{array}{c}
    x^2z = y^3 + B(\lambda)z^3 \\
    u^2w = v^3 + B(\lambda)z^3 \\
    \end{array}\right.\ .
\end{equation}
Consider the following cyclic map on $\P$
\begin{eqnarray*}
   \tau\ :&\xymatrix@1{\P(\mathcal{E})\times\P(\mathcal{E})\times\P^1\quad\ar[rr]
           && \quad\P(\mathcal{E})\times\P(\mathcal{E})\times\P^1}&  \\
   &\xymatrix@1{(x:y:z)\times(u:v:w)\times\lambda\ar@{|->}[r]
   & (x:y:z)\times(-u:\epsilon v:w)\times\lambda}&\ ,
\end{eqnarray*}
where $\epsilon$ is a primitive cubic root of unity. The second
equation in (\ref{Nami}) ensures that $\tau Y = Y$. Moreover, for any $p\in\Sing(\overline{Y})$, $\tau(p)=p$, by (\ref{cuspide}). Consider the codimension
2 diagonal locus $\Delta:= \{(x,x',\lambda)\in
\P\ |\ x=x'\}$.
Clearly $\Sing(\overline{Y})\subset\Delta$, implying that $\Sing (\overline{Y}) = \overline{Y}\cap \Delta\cap\tau\Delta$. Let then $Y$ be the strict transform of $\overline{Y}$ in the successive blow up
\[
    \xymatrix{\widehat{\P}_{\tau}\ar[r]^{\phi_{\tau\Delta}}&\widehat{\P}\ar[r]^{\phi_{\Delta}}&\P\\
                Y\ar@{^{(}->}[u]\ar[rr]^{\phi}&&\overline{Y}\ar@{^{(}->}[u]}
\]
where $\phi_{\Delta}$ is the blow up of $\P$ along $\Delta$ and $\phi_{\tau\Delta}$ is the blow up of $\widehat{\P}$ along the strict transform $\widehat{\tau\Delta}\subset\widehat{\P}$ of $\tau\Delta$.
 Since $\phi$ turns out to be a small, crepant resolution, $Y$ is
a \cy threefold and
\begin{equation}\label{nami-gt}
    \xymatrix@1{\widehat{Y}\ar@/_1pc/ @{.>}[rr]_{T}\ar[r]^{\phi}&
                Y\ar@{<~>}[r]&\widetilde{Y}}
\end{equation}
is a small \emph{non--conifold} g.t. with 6 disjoint exceptional trees of type $A_2$, where $\widetilde{Y}$ is the smooth small resolution of a bi--cubic hypersurface of type (\ref{bicubica}).

\end{example}

\subsection{Homology change induced by a small geometric transition}

We are now in a position to write down, for small g.t.'s, a
statement analogous to Theorem \ref{cambio omologico}. The
following theorem is actually a revised version of results of
Y.~Namikawa and J.H.M.~Steenbrink (see
\cite{Namikawa-Steenbrink95}, Theorem (3.2) and Example (3.8)).
After replacing Clemens' local analysis by Milnor's one, the proof
given here is completely analogous to the proof of Theorem
\ref{cambio omologico} given in \cite{R}.

\begin{theorem}\label{cambio omologico small}
Consider the small g.t. (\ref{small g.t.}). Then
\[
    h[T]=(k,0,c',c'')
\]
where
\begin{itemize}
\item $k$ is the number of homologically independent
    exceptional rational curves composing $E=\Exc(\phi)\subset Y$\ ,
\item $c'$ is the number of independent relations
    linking the homology classes of exceptional rational curves in $Y$\ ,
\item $c''$ is the number of homologically independent vanishing cycles in $\widetilde{Y}$.
\end{itemize}
Moreover,
    \begin{itemize}
    \item[(i)] the total number of irreducible components of $E$ is
    \[
    n:=\sum_{p\in P} n_p = k + c'\ ,
    \]
    where $n_p$ is the number of irreducible (rational) components
    of $E_p:=\phi^{-1}(p)$, for any $p\in P:=\Sing(\overline{Y})$,
    \item[(ii)] the \emph{global Milnor number of $\overline{Y}$} is
    \[
    m:=\sum_{p\in P} m_p = k + c''\ ,
    \]
    where $m_p$ is the Milnor number of the singular point $p\in P$. Hence $k$ turns out to be also the number of independent relations
    linking the homology classes of vanishing cycles in $\widetilde{Y}$.
    \end{itemize}
    In particular, $T$ is a type I g.t. if and only if $k'=1=k$ and if $T$ is a conifold t. then $c'=c''=c$.
\end{theorem}

\begin{remark}\label{difetto2}
As in Remark \ref{difetto}, point (b) in Definition \ref{ht} with $k,c',c''$ as in Theorem \ref{cambio omologico small},
implies that $\overline{Y}$ has defect $k$.

\noindent On the other hand, by (a) of Theorem \ref{risoluzione},
any $p\in P$ is a rational i.h.s.. Moreover, $\overline{Y}$ is
normal and $H^2(\overline{Y},\mathcal{O}_{\overline{Y}}) = 0$,
since $Y$ is a \cy 3--fold. Under all these conditions, Lemmas
(3.3) and (3.5) in \cite{Namikawa-Steenbrink95} apply to give that
\begin{equation}\label{W/C rango}
    k = \rk \left( \left\langle\text{Weil divisors of }\overline{Y}\right\rangle_{\Z} /
    \left\langle\text{Cartan divisors of }\overline{Y}\right\rangle_{\Z}\right)
\end{equation}
Recall now that a variety is called \emph{$\Q$--factorial} if any
Weil divisor is a $\Q$--Cartier divisor. Then:
\begin{itemize}
    \item \emph{given a small non trivial g.t.
    $T(Y,\overline{Y},\widetilde{Y})$, the singular 3--fold
    $\overline{Y}$ is never $\Q$--factorial.}
\end{itemize}
In fact, any primitive extremal transition reduces by 1 the rank of
$\Pic (Y)\cong H^2(Y,\Z)$. By Poinacar\'e duality on $Y$ the
exceptional cycle of a small transition can never be homologically
trivial.
\end{remark}

\begin{remark}[Example \ref{esempio-Nami} continued] If Theorem \ref{cambio omologico small} is applied to the small g.t. in Example \ref{esempio-Nami}, one finds that $k'=k=2, k''=0$, since $Y\stackrel{\phi}{\longrightarrow} \overline{Y}$ is induced by two successive blow-ups. Then $c'=n-k=12-2=10$. On the other hand the singular point (\ref{cuspide}) has Milnor number $m_p=4$, by the Milnor--Orlik Theorem (\cite{Milnor-Orlik} Theorem 1). Then the global Milnor number of $\overline{Y}$ is given by $m=6\cdot 4=24$ giving that $c''=m-k=24-2=22$. Therefore one gets the following table of Betti numbers (see \cite{R2}, Theorem 3.8):
\begin{equation*}
\begin{tabular}{lclclclc}
  \hline \hline
    \hskip11.3pt   \vrule     &\quad     &$b_2$ &\quad  & $b_3$&\quad
  &$b_4$\\
  \hline
  \hskip11.3pt\vrule&&&&&&&\\
  $Y$ \vrule   &\quad    & 21 &\quad & 8 &\quad & 21  &\\
  $\overline{Y}$ \vrule            &\quad      &  19 &\quad & 18&\quad  & 21 &\\
  $\widetilde{Y}$ \vrule &\quad  & 19 &\quad & 40 &\quad & 19 &\\
   \hline\hline
\end{tabular}
\end{equation*}
Then $\overline{Y}$ has defect 2.
\end{remark}

\begin{proof}[Proof of Theorem \ref{cambio omologico small}]
Given the small g.t. (\ref{small g.t.}), for any $p\in P$
construct $\overline{U}_p, \overline{B}_p$ and $\widetilde{U}_p,
\widetilde{B}_p$ like in Definition \ref{fibra Milnor} and Theorem \ref{i.s.topologia}. Set
\[
{U}_p:=\phi^{-1}(\overline{U}_p)\quad ,\quad
{B}_p:=\phi^{-1}(\overline{B}_p)\ .
\]
We have then the following localization, near to $p$, of the small
g.t. (\ref{small g.t.}):
\begin{equation}\label{small g.t. locale}
    \xymatrix{Y\ar@/^1pc/ @{.>}[rr]^{T}\ar[r]_{\phi}&
                \overline{Y}\ar@{<~>}[r]&\widetilde{Y}\\
              {U}_p\ar@{^{(}->}[u]\ar[r]^{\varphi_p}&
                \overline{U}_p\ar@{^{(}->}[u]\ar@{<~>}[r]&\widetilde{U}_p\ar@{^{(}->}[u]\\
              {B}_p\ar@{^{(}->}[u]\ar[r]^{\varphi_p}&
                \overline{B}_p\ar@{^{(}->}[u]\ar@{<~>}[r]&\widetilde{B}_p\ar@{^{(}->}[u]  }
\end{equation}
Let us denote:
\begin{eqnarray*}
    && U:=\bigcup_{p\in P}{U}_p \quad, \quad B:=\bigcup_{p\in P}{B}_p
    \quad, \quad Y^*:=Y\setminus B\quad , \quad U^*:=U\setminus B\  ;\\
    && \overline{U}:=\bigcup_{p\in P}\overline{U}_p \quad, \quad
    \overline{B}:=\bigcup_{p\in P}\overline{B}_p
    \quad, \quad \overline{Y}^*:=\overline{Y}\setminus \overline{B}\quad , \quad
    \overline{U}^*:=\overline{U}\setminus \overline{B}\ ;\\
    && \widetilde{U}:=\bigcup_{p\in P}\widetilde{U}_p \quad, \quad
    \widetilde{B}:=\bigcup_{p\in P}\widetilde{B}_p
    \quad, \quad \widetilde{Y}^*:=\widetilde{Y}\setminus \widetilde{B}\quad ,
    \quad \widetilde{U}^*:=\widetilde{U}\setminus \widetilde{B}\ .
\end{eqnarray*}
We get then the following Mayer--Vietoris couples:
\begin{eqnarray*}
  \mathcal{C} &:=& (Y=Y^*\cup U\ ,\ U^*=Y^*\cap U)\ , \\
  \overline{\mathcal{C}} &:=& (\overline{Y}=\overline{Y}^*\cup \overline{U}\ ,
    \ \overline{U}^*=\overline{Y}^*\cap \overline{U})\ , \\
  \widetilde{\mathcal{C}} &:=& (\widetilde{Y}=\widetilde{Y}^*\cup \widetilde{U}\ ,
    \ \widetilde{U}^*=\widetilde{Y}^*\cap \widetilde{U})\ .
\end{eqnarray*}

\begin{step I}
$\forall i\neq 2,3\ \ b_i(Y)=b_i(\overline{Y})$ and
\begin{equation*}
b_2(Y)- b_2(\overline{Y})=k \Leftrightarrow
b_3(\overline{Y})-b_3(Y)=n-k \ .
\end{equation*}
\end{step I}

Compare the singular homology long exact sequences associated with
$\mathcal{C}$ and $\overline{\mathcal{C}}$:
\begin{equation}\label{MVresol}
    \xymatrix@1{\cdots\ar[r] & H_{q}(U^*) \ar[r] &
    H_{q}(Y^*)\oplus H_{q} (U) \ar[r] & H_q (Y) \ar[r] &
    H_{q-1}(U^*)\ar[r] & \cdots}
\end{equation}
\begin{equation}\label{MVsing}
    \xymatrix@1{\cdots\ar[r] & H_{q}(\overline{U}^*) \ar[r] &
    H_{q}(\overline{Y}^*)\oplus H_{q} (\overline{U}) \ar[r] & H_q (\overline{Y}) \ar[r] &
    H_{q-1}(\overline{U}^*)\ar[r] & \cdots}
\end{equation}
Since $\phi$ is an isomorphism outside of the exceptional locus
$E$ we get
\begin{eqnarray}\label{iso-resol}
  \forall q\quad H_q(Y^*) &\cong& H_q(\overline{Y}^*) \\
 \nonumber H_q(U^*) &\cong& H_q(\overline{U}^*)
\end{eqnarray}
Moreover, by proposition \ref{contrazione i.s.}, $\overline{B}$ turns
out to be a strong deformation retract of $\overline{U}$ and, by
theorem \ref{i.s.topologia}, $\overline{B}$ is a union of cones
which can be contracted, by straight line homotopy, to $P$. Then
\begin{equation}\label{slh-sing}
    H_q(\overline{U}) \cong H_q(P) \cong \left\{ \begin{array}{cc}
      \Z^{|P|} & \text{if}\ q=0 \\
      0 & \text{otherwise.} \\
      \end{array} \right.
\end{equation}
On the other hand, recalling Proposition \ref{omotopia scoppiata}, we get the following
\begin{equation}\label{slh-resol}
    \forall q\quad H_q (U) \cong H_q(E)=\bigoplus_{p\in P} H_q(E_p)
    \cong\bigoplus_{p\in P} H_q(S^2)^{\oplus n_p}\cong \left\{ \begin{array}{cc}
      \Z^{|P|} & \text{for $q=0$} \\
      \Z^n & \text{for $q=2$} \\
      0 & \text{otherwise.} \\
    \end{array} \right.
\end{equation}
In fact, let us proceed by induction. If $n_p =
1$ then $E_p\cong S^2$. Point (b) in Theorem \ref{risoluzione}
allows us to think of $E_p$ as the union
\[
E_p = C\cup E'
\]
where $C\cong\P^1_{\C}$, $E'$ is a connected union of $n_p -1$
rational curves, whose configuration is still represented by one
of the listed graphs, and $C\cap E'$ is a single point $y$. Then
$(E',C)$ gives a Mayer--Vietoris couple. Since
\[
H_q(C\cap E')\cong H_q\left(\{ y\}\right)=0 \quad \forall q\geq 1\
,
\]
the singular homology long exact sequence of $(E',C)$ allows us to
conclude that
\[
H_q(E_p)\cong H_q(C)\oplus H_q(E') \quad \forall q\geq 2
\]
giving (\ref{slh-resol}), for $q\geq 2$, by induction hypothesis.
We have then the following exact sequence
\[
0\rightarrow H_1(C)\oplus H_1(E')\rightarrow H_1(E_p)\rightarrow
H_0\left(\{ y\}\right)\rightarrow H_0(C)\oplus H_0(E')\rightarrow
H_0(E_p)\rightarrow 0
\]
Since $\{ y\}$, $C$, $E'$ and $E_p$ are all connected,
(\ref{slh-resol}) follows for $q=1$, too.

\noindent Let us now conclude to prove Step I. By (\ref{MVresol})
and (\ref{MVsing}) and formula (\ref{slh-sing}) we get that
\[
\forall q \neq 2,3\quad b_q(Y) = b_q(\overline{Y})
\]
Moreover, the gluing of sequences (\ref{MVresol}) and
(\ref{MVsing}), by identification of isomorphic poles, reduces to
the following diagram
\begin{equation}\label{MVdiag1}
    \xymatrix{&H_3(Y)\ar[rd]&&H_2(Y^*)\oplus\Z ^n\ar[r]&H_2(Y)\ar[rd]& \\
              H_3(Y^*)\ar[ru]\ar[rd]&&H_2(U^*)\ar[ru]\ar[rd]&&&H_1(U^*) \\
              &H_3(\overline{Y})\ar[ru]&&H_2(\overline{Y}^*)\ar[r]&H_2(\overline{Y})\ar[ru]& }
\end{equation}
which gives the following relations between Betti numbers
\[
b_2(Y) - b_2(\overline{Y}) = b_3(Y)  - b_3(\overline{Y})+ n\ .
\]

\begin{step II}
$\forall q\neq 3,4\ \ b_q(\widetilde{Y})=b_q(\overline{Y})$ and
\begin{equation*}
b_3(\widetilde{Y})-b_3(\overline{Y})=c'' \Leftrightarrow
b_4(\overline{Y})-b_4(\widetilde{Y})=m-c'' \ .
\end{equation*}
\end{step II}

Compare the following singular homology long exact sequence
associated with the couple $\widetilde{\mathcal{C}}$
\begin{equation}\label{MVsmooth}
    \xymatrix@1{\cdots \ar[r]&H_q(\widetilde{U}^*)\ar[r]&H_q(\widetilde{Y}^*)\oplus H_q(\widetilde{U})\ar[r]&
    H_q(\widetilde{Y})\ar[r]&H_{q-1}(\widetilde{U}^*)\ar[r]&\cdots  }
\end{equation}
and the Mayer--Vietoris sequence (\ref{MVsing}) associated with
$\overline{\mathcal{C}}$.

\noindent Recall that, by Proposition \ref{contrazione i.s.},
$\widetilde{B}$ is a strong deformation retract of
$\widetilde{U}$. Moreover, Theorem \ref{i.s.omotopia} asserts that, for
any $p\in P$, the Milnor fiber $\widetilde{B}_p$ has the same
homology type of a bouquet of $m_p$ 3--dimensional spheres. Then
\begin{equation}\label{slh-smooth}
    H_q (\widetilde{U}) \cong H_q(\widetilde{B}) \cong
    \bigoplus_{p\in P} H_q(\widetilde{B}_p) \cong
    \left\{ \begin{array}{cc}
           \Z^{|P|} & \text{if}\ q=0 \\
           \Z^m & \text{if}\ q=3 \\
            0 & \text{otherwise.} \\
    \end{array} \right.
\end{equation}
The localization (\ref{small g.t. locale}) and the Ehresmann
fibration theorem allow us to assert that there are diffeomorphisms
\[
\overline{Y}^*\cong
\widetilde{Y}^*\quad\text{and}\quad\overline{U}^*\cong
\widetilde{U}^*
\]
Then, recalling formulas (\ref{slh-sing}) and (\ref{slh-smooth}),
we can conclude that
\[
    \forall q\neq 3,4 \quad  b_q (\widetilde{Y}) = b_q (\overline{Y})\ .
\]
Moreover, the gluing of sequences (\ref{MVsing}) and
(\ref{MVsmooth}), by identification of isomorphic poles, reduces
to the following diagram
\begin{equation}\label{MVdiag2}
  \xymatrix{&H_4(\widetilde{Y})\ar[rd]&&H_3(\widetilde{Y}^*)\oplus\Z ^m\ar[r]
             &H_3(\widetilde{Y})\ar[rd]&  \\
              H_4(\widetilde{Y}^*)\ar[ru]\ar[rd]&&H_3(\widetilde{U}^*)\ar[ru]\ar[rd]&&&H_2(\widetilde{U}^*) \\
              &H_4(\overline{Y})\ar[ru]&&H_3(\overline{Y}^*)\ar[r]&H_3(\overline{Y})\ar[ru]&}
\end{equation}
which gives the following relations between Betti numbers
\[
b_3(\widetilde{Y})-b_3(\overline{Y}) =
b_4(\widetilde{Y})-b_4(\overline{Y})+m \ .
\]

\begin{step III}
Let $k$ and $c''$ be the same parameters defined in Steps I and II
respectively. Then
\[
m=k+c''\ .
\]
\end{step III}

By Poincar\'e duality
\begin{eqnarray*}
  b_2(Y) &=& b_4(Y) \\
  b_4(\widetilde{Y}) &=& b_2(\widetilde{Y})
\end{eqnarray*}
Recall then Steps I and II to get
\begin{eqnarray*}
b_2(Y)&=&b_4(Y)=b_4(\overline{Y})=b_4(\widetilde{Y})+m-c''\\
&=&b_2(\widetilde{Y})+m-c''=b_2(\overline{Y})+m-c''=b_2(Y)-k+m-c''\\
\end{eqnarray*}
Then $m-k-c''=0$.

\begin{step IV}
$k$ is the maximal number of homologically independent exceptional
rational curves in $Y$ while $c''$ is the maximal number of
homologically independent vanishing cycles in $\widetilde{Y}$.
\end{step IV}

Since the birational contraction $\phi$ is an isomorphism outside
of the exceptional locus $E$ and by the Ehresmann fibration
theorem, we get the following composition of diffeomorphisms
\[
Y^*\cong\overline{Y}^*\cong\widetilde{Y}^*
\]
and, by Lefschetz duality,
\begin{equation}\label{rel-hom_iso}
    H_i(Y,B) \cong H^{6-i}(Y\setminus B) \cong
    H^{6-i}(\widetilde{Y}\setminus\widetilde{B}) \cong H_i(\widetilde{Y},\widetilde{B})
\end{equation}
Consider the long exact relative homology sequences of the couples
$(Y,B)$ and $(\widetilde{Y},\widetilde{B})$ and the vertical
isomorphisms given by (\ref{rel-hom_iso}):
\begin{equation}\label{rel-hom_diagram}
    \xymatrix{\cdots H_{i+1}(Y,B)\ar[r]\ar[d]^{\cong}& H_i(B)\ar[r]&
              H_i(Y)\ar[r]& H_i(Y,B)\cdots\ar[d]^{\cong}\\
              \cdots H_{i+1}(\widetilde{Y},\widetilde{B})\ar[r]&
              H_i(\widetilde{B})\ar[r]& H_i(\widetilde{Y})\ar[r]&
              H_i(\widetilde{Y},\widetilde{B})\cdots}
\end{equation}
By identifying the isomorphic poles and recalling
(\ref{slh-resol}) and (\ref{slh-smooth}) the previous long exact
sequences reduce to the following diagram:
\begin{equation}\label{k-c-diagram}
    \xymatrix{&&&&&0\ar[d]&\\
              &&&&&H_3(Y)\ar[d]&\\
              0\ar[r]&H_4(\widetilde{Y})\ar[r]&H_4(Y)\ar[r]&
              H_3(\widetilde{B})\ar[r]^{\gamma}\ar@{}[d]|{\parallel}&H_3(\widetilde{Y})\ar[r]&
              H_3(\widetilde{Y},\widetilde{B})\ar[r]\ar[d]&0\\
              &&&\Z ^m &&H_2(B)\ar@{}[r]|{=}\ar[d]^{\kappa}&\Z ^n\\
              &&&&&H_2(Y)\ar[d]&\\
              &&&&&H_2(\widetilde{Y})\ar[d]&\\
              &&&&&0&}
\end{equation}
Set
\begin{equation*}
    I:=\im [\kappa :\Z ^n = H_2(B)\longrightarrow H_2(Y)]
\end{equation*}
Then $k:= \rk (I)$ is the number of linear independent classes of
exceptional curves in $H_2(Y)$. Since
\begin{equation*}
    \xymatrix@1{0\ar[r]&I\ar[r]&H_2(Y)\ar[r]&H_2(\widetilde{Y})\ar[r]&0}
\end{equation*}
is a short exact sequence, it follows that
\begin{equation*}
    b_2(Y)=b_2(\widetilde{Y})+k
\end{equation*}
On the other hand set
\begin{equation*}
    K:=\ker [\gamma :\Z ^m\cong H_3(\widetilde{B})\longrightarrow H_3(\widetilde{Y})]
\end{equation*}
Then $m-c'':=\rk (K)$ is the number of linear independent
relations on the classes of vanishing cycles in
$H_3(\widetilde{Y})$. Since
\begin{equation*}
    \xymatrix@1{0\ar[r]&H_4(\widetilde{Y})\ar[r]&H_4(Y)\ar[r]&K\ar[r]&0}
\end{equation*}
is a short exact sequence, it follows that
\begin{equation*}
    b_4(Y)=b_4(\widetilde{Y})+m-c''
\end{equation*}

\begin{conclusion}
(i) follows from Step IV and the definition of $n$ and $c'$.

\noindent (ii) follows from Steps III and IV.

\noindent Finally (a), (b) and (c) of Definition \ref{ht} follow from Steps I, II, III and IV.

\noindent The last assertion of the statement follows from the
fact that, if $p$ is a node, then $m_p = n_p = 1$.
\end{conclusion}
\end{proof}

\begin{remark}\label{ht in omologia} Notice that Step IV in the proof of Theorem \ref{cambio omologico small} gives the following homological interpretation of $k,c',c''$
\begin{equation*}
  k = \rk(\im \kappa)\quad ,\quad  c' = \rk(\ker \kappa)\quad ,\quad
  c'' = \rk(\im \gamma)\ ,
\end{equation*}
where $\kappa$ and $\gamma$ are the homonymous maps in diagram (\ref{k-c-diagram}).

\noindent Compare with the case of a type II g.t. whose exceptional locus is a del Pezzo surface of degree $d\leq 4$, treated in the following Remark \ref{ht in omologia II}.
\end{remark}

\section{Homological type of type II geometric transition}

In the present section we will consider a \emph{type II} g.t.
\begin{equation}\label{typeII}
    \xymatrix@1{Y\ar@/_1pc/ @{.>}[rr]_T\ar[r]^{\phi}&
                \overline{Y}\ar@{<~>}[r]&\widetilde{Y}}
\end{equation}
Recalling Theorem \ref{classificazione} and Definition
\ref{type-definizione}, $\phi$ turns out to be a \emph{primitive}
contraction of an irreducible divisor $E\subset Y$ down to a point
$p\in \overline{Y}$. Let us summarize what is known.

\subsubsection{\textbf{About the singularity $p=\phi(E)$}}\label{singolarità} It is a canonical
    singularity and in particular it is a \emph{rational Gorenstein singular
    point}. The \emph{Reid invariant} $\varrho$ of $p$ (see \cite{Reid80}) can
    assume every value $1\leq \varrho\leq 8$ and
\begin{itemize}
        \item for $\varrho\leq 3$, $p$ is a i.h.s.;
        \item for $\varrho\geq 4$, $p$ has multiplicity $\varrho$ and
        minimal embedding dimension
        \[
        \dim(m_p/m_p^2)= \varrho+1\ ;
        \]
        in particular,
        for $\varrho=4$, $p$ is a complete intersection singularity and,
        for $\varrho\geq 5$, $p$ is never a complete intersection
        singular point (see \ref{dmu4} below);
    \end{itemize}

\subsubsection{\textbf{About the exceptional locus $E$ of $\phi$}}\label{luogo eccezionale} It is a \emph{generalized del Pezzo
surface} (see \cite{Reid80}, Proposition (2.13)) which is:
\begin{itemize}
    \item either $E$ is a \emph{normal} del Pezzo surface of degree $1\leq d= K_E^2\leq
    8$; in particular the degree $d$ equals the Reid invariant $\varrho$
    of $p=\phi(E)$;
    \item or $E$ is a \emph{non--normal} del Pezzo surface (see \cite{Reid94}).
\end{itemize}
Observe that the values 0 and 9 cannot be assumed by $\varrho=d$: the
former since $E$ has ample anti--canonical bundle, the latter
because (\ref{typeII}) is a transition while the contraction of a
normal del Pezzo surface of degree 9 down to a point do not admit
any smoothing: in this case $E\cong \P^2$ and $(\overline{Y},p)$
is rigid (see \cite{Altmann97} and \cite{Schlessinger71}).

\begin{remark}\label{Q-fattorialità} The contraction $\phi: Y \rightarrow \overline{Y}$ is
    the (weighted, for $d\leq 2$) blow up of the singular point $\phi
    (p)$ (see \cite{Reid80}, Theorem (2.11)). Then $\overline{Y}$ is always $\Q$--factorial, by \cite{KMM}, Proposition 5-1-6.
\end{remark}

\subsection{Normal exceptional divisor}\label{normale} A normal del Pezzo surface
$E$ occurring as exceptional locus of $\phi$ in (\ref{typeII}) is
a \emph{normal projective Gorenstein surface with ample
anti--canonical bundle}. Let $\pi:\widehat{E}\rightarrow E$ be a
minimal resolution of $E$. The following results are essentially
due to F.~Hidaka and K.--I.~Watanabe \cite{Hidaka-Watanabe81},
M.~Reid \cite{Reid80}, H.~Pinkham \cite{Pinkham77} and C.Tam\'as
\cite{Tamas04}.

\subsubsection{}\label{raz_o_ell} (\cite{Hidaka-Watanabe81} Proposition 2.1 and Theorem 2.2)
    $E$ is birationally equivalent to a ruled surface and
    \begin{itemize}
        \item either $E$ is \emph{rational},
        \item or $\widehat{E}$ is a $\P^1$--bundle over an
        elliptic curve; under the notation introduced in \cite{Hidaka-Watanabe81}, we
        will say that $E$ is \emph{elliptic}.
    \end{itemize}
    In particular $H^1(E,\mathcal{O}_E)=0$ (\cite{Hidaka-Watanabe81}, Corollary 2.5).

\subsubsection{}\label{razionale} Let $E$ be \emph{rational}. Then $E$ can assume at worst isolated
    Du Val singularities. Moreover (see \cite{Hidaka-Watanabe81}, Theorem 3.4):
    \begin{itemize}
        \item if $d=8$ then either $E\cong \P^1\times\P^1$ or
        $E$ is the cone
        over a conic in $\P^2$; in the latter case $\widehat{E}\cong
        \P(\mathcal{O}\oplus\mathcal{O}(-2))$ and $\pi$ is the
        contraction of the minimal section of $\widehat{E}$ (we have excluded
        the case $E\cong\P(\mathcal{O}\oplus\mathcal{O}(-1))$ since the
        contraction $\phi$ yields a rigid singularity
        $(\overline{Y},\phi(E))$,
        contradicting diagram (\ref{typeII}));
        \item if $1\leq d\leq 7$ then there exists a set $\Sigma$ of
        points on $\P^2$ in \emph{almost general position} (see
        \cite{Hidaka-Watanabe81}, Definition 3.2) such that
        $|\Sigma|=9-d$ and $\widehat{E}$ is the blow up
        of $\P^2$ along $\Sigma$; $\pi$ is the contraction of all
        curves on $\widehat{E}$ with self--intersection number -2.
    \end{itemize}

\subsubsection{}\label{razionaleII} If $E$ is rational then $|\Sing (E)|\leq 6$
    (see \cite{Hidaka-Watanabe81} Theorem 4.9, \cite{Pinkham77} and
    \cite{Tamas04} subsection 3.2).

\subsubsection{} \label{dmu4} If $d\geq 4$ then $E$ is rational (\cite{Gross97a} Thm. (5.2)). Moreover, the
    anti-canonical map embeds $E$ in $\P^d$ as a surface of degree $d$
    obtained by intersecting $d(d-3)/2$ hyperquadrics
    (\cite{Hidaka-Watanabe81}, Theorem 4.4(i) and Corollary 4.5(i)).

\subsubsection{}\label{ellittica} Let $E$ be \emph{elliptic}. Then $E$ can assume at worst one
    \emph{elliptic singular point} (see \cite{Reid80}, Definition
    (2.4)), $\widehat{E}\cong \P(\mathcal{O}_C\oplus\mathcal{L})$,
    where $C$ is a smooth elliptic curve and $\mathcal{L}$ is a
    positive line bundle on $C$, and $\pi$ is the contraction of
    the minimal section of $\widehat{E}$ (\cite{Hidaka-Watanabe81},
    Theorem 2.2). In particular, by \ref{dmu4}, $d=\deg E\leq 3$.

\subsubsection{}\label{d=3} If $d=3$ then the canonical map embeds $E$ as a cubic
    surface in $\P^3$(\cite{Hidaka-Watanabe81} Theorem 4.4(ii) and
    Corollary 4.5(i), \cite{Reid80} Proposition (2.3)). Then $E$ is elliptic if and only
    if it is a cone over a plain cubic curve.

\subsubsection{}\label{d=2} If $d=2$ then $E$ is isomorphic to a degree 4 hypersurface
    in the weighted projective space $\P(1,1,1,2)$. The divisor
    $-2K_E$ is very ample and the associated morphism embeds $E$
    as a degree 8 subvariety of $\P^6$. Moreover, $E$ can be
    described as a double covering of $\P^2$ ramified along a
    quartic curve without multiple components (see \ref{esempi-tII} for $d=2$). Then $E$ is
    elliptic if and only if the ramification divisor is given by
    four lines meeting in a point(\cite{Hidaka-Watanabe81} Theorem
    4.4(iii), Corollary 4.5(ii) and Proposition 4.6(i), \cite{Reid80}
    Proposition (2.3)).

\subsubsection{}\label{d=1} If $d=1$ then $E$ is isomorphic to a degree 6 hypersurface
    in the weighted projective space $\P(1,1,2,3)$. The divisor
    $-3K_E$ is very ample and the associated morphism embeds $E$
    as a degree 9 subvariety of $\P^6$. Moreover, $E$ can be
    described as a double covering of a quadratic cone
    $\mathcal{C}\subset\P^3$,
    ramified along the intersection $\mathcal{C}\cap S$, where $S$
    is a cubic surface without multiple components and not containing the
    vertex of the cone (see \ref{esempi-tII} for $d=1$). Then $E$ is elliptic if and only if the
    cubic surface $S$ is given by three planes meeting in a line which
    is tangent to $\mathcal{C}$ (\cite{Hidaka-Watanabe81} Theorem
    4.4(iv), Corollary 4.5(iii) and Proposition 4.6(ii), \cite{Reid80}
    Proposition (2.3)).

\subsection{Non--normal exceptional divisor}\label{non-normale} A non--normal del Pezzo surface
$E$ occurring as exceptional locus of $\phi$ in (\ref{typeII}) has
to satisfy the classification given in \cite{Reid94}. On the other
hand $E$ is an irreducible surface embedded in the smooth \cy
3-fold $Y$, meaning that $E$ cannot admit non--hypersurface
singularities. M.~Gross proved that these conditions imply one and only one of the
following statements
\begin{itemize}
    \item[(i)]
    Consider the Segre--del Pezzo scroll
        $\F_a := \P(\mathcal{O}_{\P^1}\oplus\mathcal{O}_{\P^1}(-a))$ embedded in $\P^{a+5}$ by means of the very ample linear system $|C_0+(a+2)f|$, where $C_0$ is the class of a  section and $f$ the class of a fiber. Then $E$ is the projection of $\F_a$ into $\P^{a+4}$ from a point in the plane spanned by the conic $C_0$ but not on $C_0$. This projection exhibits $C_0$ as a double covering of a line and makes no other identifications.
    \item[(ii)] Consider $\F_a$ embedded in $\P^{a+3}$ by means of the very ample linear system $|C_0+(a+1)f|$. Then $E$ is the projection of $\F_a$ into $\P^{a+2}$ from a point in the plane spanned by the line $C_0$ and one fiber $f$. This projection identifies $C_0$ and $f$.
\end{itemize}
In particular, a non--normal $E$ can occur only if $\deg(E)$, or equivalently the Reid
invariant $\varrho$ of $p=\phi(E)$, is equal to 7 (\cite{Gross97a}, Theorem 5.2).

\subsection{Homology change induced by a type II geometric transition}

Consider a type II g.t. (\ref{typeII}). Then $\phi$ is a primitive (in case weighted) blow up of an isolated 3--dimensional rational Gorenstein singular point $p\in\overline{Y}$, whose exceptional locus is an irreducible generalized del Pezzo surface $E\subset Y$. We get then the following result, which is the type II analogue of Theorems \ref{cambio omologico} and \ref{cambio omologico small}.

\begin{theorem}\label{cambio omologico tipo II}
The type II g.t. (\ref{typeII}) admits homological type
\[
    h[T]=(0,1,c',c'')
\]
given by
\begin{itemize}
    \item[(i)] $c'=-1 +b_2(E)-b_3(E)$\ ,
    \item[(ii)] $c''=1-\chi(\widetilde{B})=m_p-b_2(\widetilde{B})$\ ,
\end{itemize}
where $\widetilde{B}$ is the Milnor fiber of the smoothing $\widetilde{Y}$ and $m_p:=b_3(\widetilde{B})$ is the Milnor number of the unique singular point $p=\phi(E)\in\overline{Y}$. In particular $b_1(\widetilde{B})=0$.
\end{theorem}

\begin{proof} Let us start by proving (c) in Definition \ref{ht}, with
\begin{equation}\label{k,c',c'' provvisorio}
    k=1=k''\quad ,\quad c'=\chi(E) - 3\quad ,\quad c''=1-\chi(\widetilde{B})\ .
\end{equation}
Consider the same notation introduced in the proof of Theorem \ref{cambio omologico small}, but observe that now $P:=\Sing (\overline{Y})=\{p\}$. Then the singular homology long exact exact sequences of the Mayer--Vietoris couples $\mathcal{C}$ and $\overline{\mathcal{C}}$ are given by
\begin{equation}\label{MVresol-II}
    \xymatrix@1{\cdots\ar[r] & H_{q}(U^*) \ar[r] &
    H_{q}(Y^*)\oplus H_{q} (U_p) \ar[r] & H_q (Y) \ar[r] &
    H_{q-1}(U^*)\ar[r] & \cdots}
\end{equation}
\begin{equation}\label{MVsing-II}
    \xymatrix@1{\cdots\ar[r] & H_{q}(\overline{U}^*) \ar[r] &
    H_{q}(\overline{Y}^*)\oplus H_{q} (\overline{U}_p) \ar[r] & H_q (\overline{Y}) \ar[r] &
    H_{q-1}(\overline{U}^*)\ar[r] & \cdots}
\end{equation}
Then Theorem \ref{i.s.topologia} and Proposition \ref{contrazione i.s.} give that
\begin{equation}\label{slh-sing-II}
    H_q(\overline{U}_p) \cong H_q({p}) \cong \left\{ \begin{array}{cc}
      \Z & \text{if}\ q=0 \\
      0 & \text{otherwise.} \\
      \end{array} \right.
\end{equation}
On the other hand Proposition \ref{omotopia scoppiata} gives a \emph{strong deformation retraction} $U_p\approx B\approx E$, giving
\begin{equation}\label{slh-resol-II}
    \forall q\quad H_q (U_p) \cong H_q(B)\cong H_q(E)\ .
\end{equation}
Then the first equation in Definition \ref{ht}.(c) follows by comparing (\ref{MVresol-II}) and (\ref{MVsing-II}).

\noindent Consider now the singular homology long exact sequence of the Mayer--Vietoris couple $\widetilde{\mathcal{C}}$
\begin{equation}\label{MVsmooth-II}
    \xymatrix@1{\cdots \ar[r]&H_q(\widetilde{U}^*)\ar[r]&H_q(\widetilde{Y}^*)\oplus H_q(\widetilde{U}_p)\ar[r]&
    H_q(\widetilde{Y})\ar[r]&H_{q-1}(\widetilde{U}^*)\ar[r]&\cdots  }
\end{equation}
Proposition \ref{contrazione i.s.} and Theorem \ref{i.s.omotopia} then guarantee that
\begin{equation}\label{slh-smooth-II}
    \forall q\quad H_q (\widetilde{U}_p) \cong H_q(\widetilde{B})\ \Rightarrow\ \forall q\geq 4 \quad b_q (\widetilde{U}_p)=0\ .
\end{equation}
Comparing (\ref{MVsing-II}) and (\ref{MVsmooth-II}) one gets the second relation in Definition \ref{ht}.(c).

\noindent Assume now the first equation of (a) in Definition \ref{ht} giving $b_2(\overline{Y})=b_2(\widetilde{Y})$ and notice that $b_1(\widetilde{Y})=b_1(Y)=0$, since $Y$ and $\widetilde{Y}$ are \cy 3--folds. The latter gives in particular that $b_1(\overline{Y})=0$, as a consequence of the Leray spectral sequence of the sheaf $\phi_*\Z$. Then the comparison of (\ref{MVsing-II}) to (\ref{MVsmooth-II}) and the Five Lemma give $H_1(\widetilde{B})\cong H_1(P)=0$, proving (ii).

\noindent The proof of (a) and (b) in Definition \ref{ht}, with $k=1$ and $c',c''$ as in the statement, are postponed to subsections \ref{caso normale-razionale}, \ref{caso ellittico} and \ref{caso non-normale} where these relations will be analyzed, case by case, under stronger assumptions. In particular, (i) follows by (\ref{omologia eccezionale-sing}), (\ref{rel-ellittiche}) and (\ref{omologia eccezionale-sing-nn}).
\end{proof}

\subsection{The normal and rational case}\label{caso normale-razionale} Given the type II g.t. (\ref{typeII}) let us now consider the case in which $E$ is assumed to be a \emph{normal and rational del Pezzo surface}. This hypothesis includes all the cases in which $E$ is smooth and it is not so restrictive since it leaves out only the following two further cases:
\begin{itemize}
  \item $E$ is a \emph{normal and elliptic} del Pezzo surface,
  \item $E$ is a \emph{non-normal} del Pezzo surface,
\end{itemize}
which will be discussed in the following subsections \ref{caso ellittico} and \ref{caso non-normale}, respectively.

\begin{theorem}\label{cambio omologico II +}
Assume that the type II g.t. (\ref{typeII}) admits a \emph{normal and rational} del Pezzo surface as exceptional divisor $E=\Exc(\phi)$ and let $\pi : \widehat{E}\rightarrow E$ be a minimal resolution of $E$, with $L:=\Exc(\pi)$ which is composed by $n_E$ rational curves organized in at most six A--D--E trees, by \ref{razionale} and \ref{razionaleII}. Then it admits homological type $h[T]=(0,1,c',c'')$ as in Theorem \ref{cambio omologico tipo II} and moreover
\begin{itemize}
    \item[(i)] $c'=9-d-n_E$\ ,
    \item[(ii)] if $d\leq 4$ then $c''=m_p$\ .
\end{itemize}
\end{theorem}

Let us postpone the proof of the previous theorem to recall that $\overline{Y}$ is $\Q$-factorial, by Remark \ref{Q-fattorialità}. This fact allows us to conclude the following

\begin{lemma}\label{difetto=0}
Let $p\in \overline{Y}$ be the 3--dimensional i.s. obtained as image of a (generalized) del Pezzo surface $E\subset Y$ under a type II birational contraction from a \cy threefold $Y$, as above. Then
\begin{equation*}
    b_4(\overline{Y}) = b_2(\overline{Y})\ .
\end{equation*}
\end{lemma}

\begin{proof}
If $p\in \overline{Y}$ is an i.h.s. then the statement is a consequence of Lemmas (3.3) and (3.5) in \cite{Namikawa-Steenbrink95}. Actually we will observe that Namikawa--Steenbrink considerations applies to the more general case of $p\in\overline{Y}$, too.

\noindent Following their notation, notice first of all that, since $p\in \overline{Y}$ is a rational singular point, $Weil(\overline{Y})/Cart(\overline{Y})$ is a finitely generated abelian group and let $\sigma(\overline{Y})$ denote its rank. Moreover, the rationality of $p$ gives $h^2(\mathcal{O}_{\overline{Y}})=0$, by Leray spectral sequence. Then Lemma (3.5) in \cite{Namikawa-Steenbrink95} applies to $p\in\overline{Y}$: in fact arguments proving this Lemma are local cohomology exact sequence, Leray spectral sequence and Goresky--MacPherson theorem \cite{Steenbrink83} Theorem (1.11) holding for an i.s. By applying Lemma (3.3) in \cite{Namikawa-Steenbrink95} we get that $\sigma(\overline{Y})=b_2(\overline{Y})-b_4(\overline{Y})$. Then $\Q$--factoriality of $\overline{Y}$ ends up the proof.
\end{proof}

\begin{proof}[Proof of Theorem \ref{cambio omologico II +}]
By \ref{razionale} and \ref{razionaleII}, $E=\Exc(\phi)$ admits at worst six isolated Du Val singularities. Moreover, a minimal resolution $\pi:\widehat{E}\rightarrow E$ is the contraction of all curves in $\widehat{E}$ with self--intersection number $-2$ where $\widehat{E}$ is either $\P^1\times\P^1$ (when $d=8$) or $\P^2$ blown up in $9-d$ points in almost general position. Define $Q:=\Sing(E)$ and consider the exceptional tree $L_q:=\pi^{-1}(q)$, for any $q\in Q$. Saying $n_q$ the number of irreducible components of $L_q$ consider the global number $n_E:=\sum_{q\in Q}n_q$. The same induction argument proving (\ref{slh-resol}) gives $n_E=b_2(L)$. Consider the relative homology of the couple $(\widehat{E},L)$ and define $\kappa_E:H_2(L,\Q)\rightarrow H_2(\widehat{E},\Q)$. Call
$$k_E:=\rk[\im(\kappa_E)]\quad,\quad c_E:=\rk[\ker(\kappa_E)]\ .$$
Then $n_E=k_E+c_E$ and we claim that
\begin{equation}\label{omologia eccezionale-sing}
    b_0(E)=b_4(E)=1\ ,\ b_1(E)=0\ ,\ b_3(E)=c_E\ ,\ b_2(E)=10-d-k_E\  .
\end{equation}
In fact, first equalities on the left of (\ref{omologia eccezionale-sing}) are clearly obvious. Then apply the local analysis in \S \ref{sing-locale} to any $q\in Q\subset E$ and define
\begin{eqnarray*}
    && U_Q:=\bigcup_{q\in Q}{U}_q \quad, \quad B_Q:=\bigcup_{q\in Q}{B}_q
    \quad, \quad \widehat{E}^*:=\widehat{E}\setminus B_Q\quad , \quad U_Q^*:=U_Q\setminus B_Q\  ;\\
    && \overline{U}_Q:=\bigcup_{q\in Q}\overline{U}_q \quad, \quad
    \overline{B}_Q:=\bigcup_{q\in Q}\overline{B}_q
    \quad, \quad E^*:=E\setminus \overline{B}_Q\quad , \quad
    \overline{U}^*_Q:=\overline{U}_Q\setminus \overline{B}_Q\ .
\end{eqnarray*}
Then compare the associated Mayer--Vietoris homology sequences:
\begin{equation}\label{MV-exc-resol}
    \xymatrix@1{\cdots\ar[r] & H_{i}(U_Q^*) \ar[r] &
    H_{i}(\widehat{E}^*)\oplus H_{i} (U_Q) \ar[r] & H_i (\widehat{E}) \ar[r] &
    H_{i-1}(U_Q^*)\ar[r] & \cdots}
\end{equation}
\begin{equation}\label{MV-exc-sing}
    \xymatrix@1{\cdots\ar[r] & H_{i}(\overline{U}_Q^*) \ar[r] &
    H_{i}(E^*)\oplus H_{i} (\overline{U}_Q) \ar[r] & H_i (E) \ar[r] &
    H_{i-1}(\overline{U}_Q^*)\ar[r] & \cdots}
\end{equation}
Going on precisely as in Step I of the Proof of the Theorem \ref{cambio omologico small} and recalling that $\widehat{E}$ is either $\P^1\times\P^1$ or the blow up of $\P^2$ in $9-d$ in almost general position, one gets the following relation
\begin{equation}\label{rel-exc}
    b_4(\widehat{E})-n_E+10-d=b_4(E)-b_3(E)+b_2(E)-b_1(E)\ .
\end{equation}
On the other hand, compare the relative homology sequences of the couples $(E,Q)$ and $(\widehat{E},L)$ as follows
\begin{equation}\label{rel-hom-exc}
    \xymatrix@-.5pc{&&0\ar[d]&&&&\\
        &&H_4(\widehat{E})\ar[d]&&&H_3(\widehat{E})\ar[d]&\\
        0\ar[r]&H_4(E)\ar[r]&H_4(E,Q)\ar[d]\ar[r]&H_3(Q)\ar[r]&H_3(E)\ar[r]&H_3(E,Q)\ar[d]\ar[r]&H_2(Q)
        \ar[dddll]|!{[ddd];[ll]}\hole\\
    &&H_3(L)\ar[uurrr]|!{[uu];[rrr]}\hole &&&H_2(L)\ar[d]^<<<<<<{\kappa_E}&\\
    &&&&&H_2(\widehat{E})\ar[d]&\\
    &&&&H_2(E)\ar[r]&H_2(E,Q)\ar[d]\ar[r]&H_1(Q)\\
    &&&&&H_1(L)&}
\end{equation}
Observe that $b_3(L)=b_1(L)=b_3(Q)=b_2(Q)=b_1(Q)=0$. Hence $H_4(E)\cong H_4(E,Q)\cong H_4(\widehat{E})$, giving
\begin{equation}\label{b4-exc}
    b_4(\widehat{E})=b_4(E)\ ,
\end{equation}
and $H_i(E)\cong H_i(E,Q)$ for $i=2,3$. Then the right vertical sequence in (\ref{rel-hom-exc}) translates immediately in the following exact sequence
\begin{equation}\label{rel-hom-exc_2}
    \xymatrix{0\ar[r]&H_3(\widehat{E})\ar[r]&H_3(E)\ar[r]&H_2(L)\ar[r]^-{\kappa_E}&H_2(\widehat{E})\ar[r]&H_2(E)\ar[r]&0}\ .
\end{equation}
Recalling that $b_2(L)=n_E$, $b_3(\widehat{E})=0$ and $b_2(\widehat{E})=10 - d$ we get
\begin{equation}\label{rel-exc_2}
    b_3(E)-b_2(E)=n_E-10+d\ .
\end{equation}
The comparison of (\ref{rel-exc}),(\ref{b4-exc}) and (\ref{rel-exc_2}) then gives $b_1(E)=0$.

\noindent Moreover, $b_3(\widehat{E})=0$ in (\ref{rel-hom-exc_2}) gives that $b_3(E)=\rk[\ker (\kappa_E)]=c_E$.

\noindent The last equality in (\ref{omologia eccezionale-sing}) is then obtained by recalling that $n_E=k_E+c_E$.

\noindent Let us now prove relations in Definition \ref{ht}.(a). First of all recall Mayer--Vietoris exact sequences (\ref{MVresol-II}), (\ref{MVsing-II}) and (\ref{MVsmooth-II}) of couples $\mathcal{C},\overline{\mathcal{C}}$ and $\widetilde{\mathcal{C}}$. Relations (\ref{omologia eccezionale-sing}) and the Five Lemma prove that
\begin{equation}\label{relazioni}
    b_1(Y)=b_1(\overline{Y})\quad , \quad b_5(\overline{Y})=b_5(\widetilde{Y})\ .
\end{equation}
But $b_1(Y)=b_1(\widetilde{Y})=0$, since $Y$ and $\widetilde{Y}$ are \cy threefolds. Then (\ref{relazioni}) and Poincar\'{e} duality prove the first line in Definition \ref{ht}.(a).

\noindent For the second line in Definition \ref{ht}.(a), observe that Lefschetz duality and Ehresmann Fibration Theorem give
\[
    H_i(Y,B)\cong H^{6-i}(Y^*)\cong H^{6-i}(\widetilde{Y}^*)\cong H_i(\widetilde{Y},\widetilde{B})\ .
\]
Then isomorphisms (\ref{slh-resol-II}) and the vanishing in (\ref{omologia eccezionale-sing}) reduce relative homology long exact sequences of couples $(Y,B)$ and $(\widetilde{Y},\widetilde{B})$ to give the following type II version of diagram (\ref{k-c-diagram}), where vertical sequences are given by relative homology of couple $(Y,B)$, while horizontal ones are obtained by relative homology of $(\widetilde{Y},\widetilde{B})$:
\begin{equation}\label{k-c-diagram-II}
    \xymatrix@-.5pc{&&H_4(E)\ar[d]^{\lambda}&&&&\\
              &&H_4(Y)\ar[d]&&&H_3(Y)\ar[d]&\\
              0\ar[r]&H_4(\widetilde{Y})\ar[r]&H_4(\widetilde{Y},\widetilde{B})\ar[d]\ar[r]&
              H_3(\widetilde{B})\ar[r]^{\gamma}&H_3(\widetilde{Y})\ar[r]&
              H_3(\widetilde{Y},\widetilde{B})\ar[d]\ar[r]& H_2(\widetilde{B})\ar[dddll]|!{[ddd];[ll]}\hole\\
              &&H_3(E)\ar[uurrr]|!{[uu];[rrr]}\hole&&&H_2(E)\ar[d]^(0.6){\kappa}&\\
              &&&&&H_2(Y)\ar[d]&\\
              &&&&H_2(\widetilde{Y})\ar[r]&H_2(\widetilde{Y},\widetilde{B})\ar[r]\ar[d]&H_1(\widetilde{B})\\
              &&&&&0&}
\end{equation}
The vertical sequence on the left of this diagram gives the following short exact sequence
\begin{equation*}
   \xymatrix@1{0\ar[r]&\im (\lambda)\ar[r]&H_4(Y)\ar[r]&H_4(\widetilde{Y},\widetilde{B})\ar[r]&0}\ .
\end{equation*}
Observe that $b_4(E)=1$ and $E\subset Y$ is the exceptional divisor of a blow up, giving $\rk(\im \lambda)=1$.
Therefore
\begin{equation}\label{b4 II'}
    b_4(Y)-b_4(\widetilde{Y},\widetilde{B})= 1\ .
\end{equation}
Since $\widetilde{Y}$ is a \cy threefolds, $h^1(\mathcal{O}_{\widetilde{Y}})=h^2(\mathcal{O}_{\widetilde{Y}})=0$ and the \emph{exponential sequence} gives $\Pic(\widetilde{Y})\cong H^2(\widetilde{Y},\Z)$. On the other hand, by \cite{Kollar-Mori92} (12.2.1.3) and (12.2.1.4.2), we are in a position to apply Proposition 3.1 in \cite{Gross97b} implying that the Picard number remains invariant when smoothing $\overline{Y}$, which is
\begin{equation}\label{Pic-invariante}
    \rho(\overline{Y})=\rho(\widetilde{Y})=b_4(\widetilde{Y})\ .
\end{equation}
Since $p\in\overline{Y}$ is a rational singularity, pushing forward the exponential sequence for $Y$ induces the following exact sequence on $\overline{Y}$
\begin{equation}\label{sing-exp-sequence}
    \xymatrix@1{0\ar[r]&\Z \ar[r]&\mathcal{O}_{\overline{Y}}\ar[r]&\mathcal{O}_{\overline{Y}}^*\ar[r]&R^1\phi_*\Z\ar[r]&0}
\end{equation}
On the other hand, the Leray spectral sequence converging to $H^i(Y,\Z)$ gives the following lower terms exact sequence
\begin{equation*}
    \xymatrix@1{0\ar[r]&H^1(\overline{Y},\Z)\ar[r]&H^1(Y,\Z)\ar[r]&H^0(\overline{Y},R^1\phi_*\Z)\ar[r]&
    H^2(\overline{Y},\Z)\hskip 2pt \ar@<-.5ex>@{^{(}->}[r]&H^2(Y,\Z)}
\end{equation*}
where the latter injection is given by the fact that $\phi:Y\rightarrow\overline{Y}$ is a blow up centered in a point.
Moreover, $Y$ is a \cy threefold, giving
$$H^1(Y,\Z)=0\ \Longrightarrow\  H^1(\overline{Y},\Z)=0\ .
$$
Notice that $R^1\phi_*\Z$ is a skyscraper sheaf supported on $p\in \overline{Y}$, then
\begin{equation*}
    H^0(R^1\phi_*\Z)=0\quad\Longleftrightarrow\quad R^1\phi_*\Z=0\ ,
\end{equation*}
implying that (\ref{sing-exp-sequence}) actually gives the following \emph{exponential sequence for $\overline{Y}$}
\begin{equation}\label{sing-short-exp-sequence}
    \xymatrix@1{0\ar[r]&\Z \ar[r]&\mathcal{O}_{\overline{Y}}\ar[r]&\mathcal{O}_{\overline{Y}}^*\ar[r]&0}\ .
\end{equation}
On the other hand, the Leray spectral sequence converging to $H^i(Y,\mathcal{O}_Y)$ gives the following lower terms exact sequence
\begin{equation*}
    \xymatrix@1{0\ar[r]&H^1(\mathcal{O}_{\overline{Y}})\ar[r]&H^1(\mathcal{O}_Y)\ar[r]&H^0(\overline{Y},R^1\phi_*\mathcal{O}_Y)\ar[r]&
    H^2(\mathcal{O}_{\overline{Y}})\ar[r]&H^2(\mathcal{O}_Y)}
\end{equation*}
where $h^1(\mathcal{O}_{Y})=h^2(\mathcal{O}_Y)=0$ and $R^1\phi_*\mathcal{O}_Y=0$, since $Y$ is a \cy threefold and $p\in\overline{Y}$ is a rational singular point. Therefore
\begin{equation*}
    h^1(\mathcal{O}_{\overline{Y}})=h^2(\mathcal{O}_{\overline{Y}})=0
\end{equation*}
giving, by (\ref{sing-short-exp-sequence}), that
\begin{equation}\label{Pic-sing}
    \Pic (\overline{Y})\cong H^1(\mathcal{O}_{\overline{Y}}^*)\cong H^2(\overline{Y},\Z)\ .
\end{equation}
Recall that $\phi$ is a primitive contraction, which is $\rho(Y)-\rho(\overline{Y})=1$. Then (\ref{Pic-invariante}), (\ref{Pic-sing}), Lemma \ref{difetto=0} and Poincar\'{e} Duality prove the second and the fourth lines in Definition \ref{ht}.(a) with $k'=0$ and $k=k''=1$.

\noindent Therefore (\ref{b4 II'}) and diagram (\ref{k-c-diagram-II}) give that $H_4(\widetilde{Y})\cong H_4(\widetilde{Y},\widetilde{B})$ (guaranteeing the injectivity of $\gamma$).

\noindent Recall Theorem \ref{cambio omologico tipo II} and in particular equations in Definition \ref{ht}.(c) with $k,c',c''$ as in (\ref{k,c',c'' provvisorio}). By (\ref{omologia eccezionale-sing}), these equations can be now rewritten as follows
\begin{eqnarray}\label{b3}
  b_3(Y) &=& b_3(\overline{Y}) +1 - b_2(E)+b_3(E) = b_3(\overline{Y}) -9+d+n_E  \\
  \nonumber
  b_3(\overline{Y}) &=& b_3(\widetilde{Y}) + \chi(\widetilde{B})-1\ ,
\end{eqnarray}
proving relations in the third line of Definition \ref{ht}.(a) with $c'=9-d-n_E$, as in point (i) of the statement. Equations in Definition \ref{ht}.(b) then follow immediately by those in Definition \ref{ht}.(a) and \cy conditions on $Y$ and $\widetilde{Y}$.

\noindent Let us now assume that $d\leq 4$: then (ii) is a consequence of Theorem \ref{i.s.omotopia}, since the Milnor fiber $\widetilde{B}$ has the homotopy type of a bouquet of 3--spheres and
\begin{equation}\label{omologia milnor}
    b_0(\widetilde{B})=1\ ,\ b_1(\widetilde{B})=b_2(\widetilde{B})=0\ ,\ b_3(\widetilde{B})= m_p\ .
\end{equation}
\end{proof}

\begin{remark}\label{ht in omologia II}
\emph{If the exceptional del Pezzo surface $E$ has degree $d\leq 4$} then $k,c',c''$ admit \emph{the same} homological interpretation given in Remark \ref{ht in omologia} for a small g.t.

\noindent In fact, (\ref{omologia milnor}) applied to diagram (\ref{k-c-diagram-II}) gives $H^2(\widetilde{Y})\cong H^2(\widetilde{Y},\widetilde{B})$ and the following short exact sequence
\[
    \xymatrix@1{0\ar[r]&\im (\kappa)\ar[r]&H_2(Y)\ar[r]&H_2(\widetilde{Y})\ar[r]&0}\ .
\]
On the other hand, the second equation in Definition \ref{ht}.(a) with $k=1$ gives
\[
    k=1=b_2(Y)-b_2(\widetilde{Y})=\rk(\im \kappa)\ .
\]
Recalling then (i) and (\ref{omologia eccezionale-sing}) we get $k+c'=10-d-k_E=b_2(E)$ giving necessarily that
\[
    c'=b_2(E)-\rk(\im \kappa)=\rk(\ker \kappa)\ .
\]
Finally (ii) in Theorem \ref{cambio omologico II +} and the injectivity of $\gamma$ in diagram (\ref{k-c-diagram-II}) imply that
$$c''=\rk(\im(\gamma))\ .$$
\end{remark}

\begin{remark}\label{Milnor-minorazioni} \emph{If the exceptional del Pezzo surface $E$ is smooth} then, comparing relations (i) in Theorem \ref{cambio omologico II +} and (ii) in Theorem \ref{cambio omologico tipo II},  with Theorem 3.3 and Remark 3.6 in \cite{KK}, we get the following \emph{conditions on the Milnor number $m_p$ of the singular point }$p\in \overline{Y}$:
\begin{eqnarray}
  \label{d=1_milnor}\text{if $d=1$ then} && m_p = 50\ , \\
  \label{d=2_milnor}\text{if $d=2$ then} && m_p = 27\ , \\
  \label{d=3_milnor}\text{if $d=3$ then} && m_p = 16\ , \\
  \label{d=4_milnor}\text{if $d=4$ then} && m_p = 9\ , \\
  \text{if $d=5$ then} && m_p=b_2(\widetilde{B})+4\ , \\
  \label{d=6_milnor}
  \text{if $d=6$ then} && m_p=\left\{\begin{array}{c}
                                         b_2(\widetilde{B}) -1\\
                                         b_2(\widetilde{B})+1
                                       \end{array}
  \right.\ , \\
  \label{d=7_milnor}
  \text{if $d=7$ then} && m_p=b_2(\widetilde{B})\\
  \label{d=8_milnor}
  \text{if $d=8$ then} && m_p=b_2(\widetilde{B})+1
\end{eqnarray}
Since $p\in\overline{Y}$ is a singular point, (\ref{d=7_milnor}) and the first case in (\ref{d=6_milnor}) prove that $b_2(\widetilde{B})>0$, showing that \emph{the Milnor fiber $\widetilde{B}$ cannot admit the homotopy type of a bouquet of 3--dimensional spheres when $d=7$ and when $d=6$ with $c=1$.}

\noindent Notice that $c+k=c+1$ gives the following (dual) Coxeter numbers
\begin{equation}\label{coxeter}
c+k=\left\{\begin{array}{cc}
               30 = \text{Coxeter}(E_8) & \text{for $d=1$} \\
               18 = \text{Coxeter}(E_7) & \text{for $d=2$}\\
               12= \text{Coxeter}(E_6) & \text{for $d=3$}\\
               8= \text{Coxeter}(D_5) & \text{for $d=4$}\\
               5= \text{Coxeter}(A_4) & \text{for $d=5$}\\
               2= \text{Coxeter}(A_1) & \text{for $d=8$}
             \end{array}
\right.
\end{equation}
as argued in \cite{Morrison-Vafa96} \S 7.1 and in \cite{Morrison-Seiberg97} \S 3. Unfortunately this is not true for $d=6$ and $d=7$, suggesting \emph{a relation between the vanishing of $b_2(\widetilde{B})$ and the existence of a Coxeter group whose number gives $c+k$}.
\end{remark}

\subsubsection{Examples}\label{esempi-tII}
The Reid and Hidaka--Watanabe results (\cite{Reid80} Proposition (2.13) and \cite{Hidaka-Watanabe81} Theorem 4.4), here reported in \ref{dmu4}, \ref{d=3}, \ref{d=2} and \ref{d=1}, allows us to easily construct examples of type II g.t's admitting Milnor numbers as in (\ref{d=4_milnor}), (\ref{d=3_milnor}), (\ref{d=2_milnor}) and (\ref{d=1_milnor}), respectively (see also Theorem 4.5 in \cite{KK} for a comparison).

\halfline $d=3$. Consider the generic quintic 3-fold in $\P^4$ admitting a triple point in $p:=[1:0:\cdots :0]\in\P^4$
\[
  \overline{Y} := \left\{x_0^2f_3(x_1,\ldots,x_4)+x_0f_4(x_1,\ldots,x_4)+f_5(x_1,\ldots,x_4)=0\right\}\ ,
\]
where $f_k$ are homogeneous polynomials of degree $k$.

\noindent Clearly the generic quintic 3-fold $\widetilde{Y}\subseteq\P^4$ is a smoothing for $\overline{Y}$.

\noindent Blow up $\P^4$ in $p$ and let $\P^3[l_1:\ldots :l_4]$ be the exceptional divisor. Then the strict transform $Y$ of $\overline{Y}$ admits exceptional locus given by the following cubic surface $E:=\left\{f_3(l_1,\ldots ,l_3)=0\right\}\subseteq\P^3[l_1:\ldots :l_4]$, which is smooth for generic $f_3$ (compare with \ref{d=3}). Observe that $Y$ is a \cy threefold since $h^i(\mathcal{O}_{Y})= h^i(\mathcal{O}_{\overline{Y}})$ (\cite{BPvdV84} Theorem I.(9.1)) and $Y\rightarrow \overline{Y}$ is a crepant resolution (\cite{Reid80} Theorem (2.11) and Corollary (2.12)). Therefore
\begin{itemize}
  \item $T(Y,\overline{Y},\widetilde{Y})$ \emph{is generically a type II g.t. whose exceptional divisor $E$ is a smooth del Pezzo surface of degree 3}.
\end{itemize}
Since $f_3=0$ is the local equation of the triple point $p\in\overline{Y}$ and it is a homogeneous polynomial, the Milnor--Orlik Theorem (\cite{Milnor-Orlik} Theorem 1) allows us to conclude that the Milnor number of $p$ is given by $m_p=(3-1)^4=16$.

\noindent Since $\widetilde{Y}$ is a projective hypersurface, its Betti numbers are well known and Theorem \ref{cambio omologico II +} gives $c'=6$, $c''=16$ and the following table
\begin{equation*}
\begin{tabular}{lclclclc}
  \hline \hline
    \hskip11.4pt   \vrule     &\quad     &$b_2$ &\quad  & $b_3$&\quad
  &$b_4$\\
  \hline
  \hskip11.4pt\vrule&&&&&&&\\
  $Y$ \vrule   &\quad    & 2 &\quad & 182 &\quad & 2  &\\
  $\overline{Y}$ \vrule            &\quad      &  1 &\quad & 188&\quad  & 1 &\\
  $\widetilde{Y}$ \vrule &\quad  & 1 &\quad & 204 &\quad & 1 &\\
   \hline\hline
\end{tabular}
\end{equation*}

\halfline $d=2$. Consider the degree 6 hypersurface of $\P(1,1,1,1,2)[x_0,x_1,x_2,x_3,y]$ given by
\[
  \overline{Y} := \left\{x_0^2f_4(x_1,\ldots,x_3,y)+x_0f_5(x_1,\ldots,x_3,y)+f_6(x_1,\ldots,x_3,y)=0\right\}\ ,
\]
where $f_k$ is a generic weighted homogeneous polynomials of degree $k$. Clearly $p:=[1:0:\cdots :0]$ is an i.h.s. for $\overline{Y}$ whose local equation is given by the w.h. polynomial $f_4=0$. Then, again by Milnor--Orlik theorem, $m_p=(4-1)^3=27$.

\noindent The strict transform $Y$ of $\overline{Y}$ under the weighted blow up of $\P(1,1,1,1,2)$ in $p$ admits exceptional locus given by the degree 4 surface $E:=\{f_4=0\}\subseteq\P(1,1,1,2)$ which is smooth for $f_4$ sufficiently general.

\noindent Finally the generic degree 6 hypersurface of $\widetilde{Y}\subset\P(1,1,1,1,2)$ is smooth since the latter has a unique isolated singular point. Then
\begin{itemize}
  \item $T(Y,\overline{Y},\widetilde{Y})$ \emph{is generically a type II g.t. whose exceptional divisor $E$ is a smooth del Pezzo surface of degree 2}.
\end{itemize}
Observe that $f_4(x_1,x_2,x_3,y)=y^2+yg_2(\mathbf{x})+g_4(\mathbf{x})$ with $g_l$ generic homogeneous polynomials of degree $l$ in $\mathbf{x}=x_1,x_2,x_3$. Then $E$ turns out to be a double covering of $\P^2[\mathbf{x}]$ ramified along the discriminant quartic plane curve $\Delta=\{g_2^2-4g_4=0\}$ (compare with \ref{d=2}).

\noindent Since $\widetilde{Y}$ is a weighted projective hypersurface, its Betti numbers are well known (\cite{Dolgachev} 4.3.2, \cite{Steenbrink87}, \cite{Iano-Fletcher} Thm. 7.2) and Theorem \ref{cambio omologico II +} gives $c'=7$, $c''=27$ and the following table
\begin{equation}\label{tabella_d=2}
\begin{tabular}{lclclclc}
  \hline \hline
    \hskip11.4pt   \vrule     &\quad     &$b_2$ &\quad  & $b_3$&\quad
  &$b_4$\\
  \hline
  \hskip11.4pt\vrule&&&&&&&\\
  $Y$ \vrule   &\quad    & 2 &\quad & 174 &\quad & 2  &\\
  $\overline{Y}$ \vrule            &\quad      &  1 &\quad & 181&\quad  & 1 &\\
  $\widetilde{Y}$ \vrule &\quad  & 1 &\quad & 208 &\quad & 1 &\\
   \hline\hline
\end{tabular}
\end{equation}

\halfline $d=1$. Consider the degree 8 hypersurface of $\P(1,1,1,2,3)[x_0,x_1,x_2,y,z]$ given by
\[
    \overline{Y} := \left\{x_0^2f_6(x_1,x_2,y,z)+x_0f_7(x_1,x_2,y,z)+f_8(x_1, x_2,y,z)=0\right\}\ ,
\]
where $f_k$ are generic weighted homogeneous polynomials of degree $k$. As before $p:=[1:0:\cdots :0]$ is an i.h.s. for $\overline{Y}$ whose local equation is given by the w.h. polynomial $f_6=0$. Then $m_p=(3-1)(6-1)^2=50$.

  \noindent The strict transform $Y$ of $\overline{Y}$ under the weighted blow up of $\P(1,1,1,2,3)$ in $p$ admits exceptional locus given by the degree 6 surface $E:=\{f_6=0\}\subseteq\P(1,1,2,3)$ which is smooth for $f_6$ sufficiently general.

  \noindent The generic degree 8 hypersurface of $\widetilde{Y}\subset\P(1,1,1,2,3)$ is still smooth since the latter has only isolated singular points. Then
\begin{itemize}
  \item $T(Y,\overline{Y},\widetilde{Y})$ \emph{is generically a type II g.t. whose exceptional divisor $E$ is a smooth del Pezzo surface of degree 1}.
\end{itemize}
Observe that $f_6(x_1,x_2,y,z)=z^2+zg_3(\mathbf{x},y)+g_6(\mathbf{x},y)$ with $g_l$ generic weighted homogeneous polynomial of degree $l$ in $\mathbf{x}=x_1,x_2$ and $y$. Then $E$ turns out to be a double covering of $\P[1,1,2]$ ramified along the discriminant degree 6 plane curve $\Delta=\{g_3^2-4g_6=0\}$. Let us first of all observe that
\[
    g_3(\mathbf{x},y)=h_1(\mathbf{x})\cdot y + h_3(\mathbf{x})\quad ,\quad g_6(\mathbf{x},y)=ay^3+h_2(\mathbf{x})\cdot y^2+ h_4(\mathbf{x})\cdot y+ h_6(\mathbf{x})
\]
where $h_l$ is a generic homogeneous polynomials of degree $l$. Then the discriminant curve $\Delta$ do not pass through the unique singular point of $\P[1,1,2]$, which is $[0:0:1]$. Moreover, the linear system associated with $\co(2)$ embeds $\P(1,1,2)$ as a quadratic cone $\mathcal{C}\subset\P^3$, translating the degree 6 equation of $\Delta$ in the equation of a cubic surface of $\P^3$ cutting $\Delta$ on $\mathcal{C}$ outside of its vertex (compare with \ref{d=1}).

\noindent Since $\widetilde{Y}$ is a weighted projective hypersurface, its Betti numbers are well known (\cite{Dolgachev} 4.3.2, \cite{Steenbrink87}, \cite{Iano-Fletcher} Thm. 7.2)
and Theorem \ref{cambio omologico II +} gives $c'=8$, $c''=50$ and the following table
\begin{equation*}
\begin{tabular}{lclclclc}
  \hline \hline
    \hskip11.4pt   \vrule     &\quad     &$b_2$ &\quad  & $b_3$&\quad
  &$b_4$\\
  \hline
  \hskip11.4pt\vrule&&&&&&&\\
  $Y$ \vrule   &\quad    & 2 &\quad & 156 &\quad & 2  &\\
  $\overline{Y}$ \vrule            &\quad      &  1 &\quad & 164&\quad  & 1 &\\
  $\widetilde{Y}$ \vrule &\quad  & 1 &\quad & 214 &\quad & 1 &\\
   \hline\hline
\end{tabular}
\end{equation*}

\halfline $d=4$. Finally let us consider the case of a type II g.t. whose exceptional divisor is a del Pezzo surface of degree 4, meaning that $\Sing(\overline{Y})$ is composed by a unique \emph{complete intersection and non--hypersurface singularity}. At this purpose consider the degree 6 complete intersection in $\P^5$ given by
\begin{equation*}
    \overline{Y} := X_1\cap X_2\quad \text{with}\quad\begin{array}{ccc}
                                                    X_1 & := & \left\{x_0f_2(\mathbf{x})+f_3(\mathbf{x})=0\right\} \\
                                                    X_2 & := & \left\{x_0g_2(\mathbf{x})+g_3(\mathbf{x})=0\right\}
                                                  \end{array}
\end{equation*}
where $f_k,g_k$ are generic homogeneous polynomials of degree $k$ in $\mathbf{x}=x_1,\ldots,x_5$. As before $\Sing(\overline{Y})=\{p\}$, where $p:=[1:0:\ldots:0]$. In particular, $p$ is a \emph{quadratic 3--dimensional c.i.s.}, meaning that it is locally described by the germ of singularity given by the zero locus in $\C^5$ of two homogeneous quadratic polynomials, namely $f_2$ and $g_2$. Then, recalling \cite{Looijenga} Example 1 in (5.11), the Milnor number of $p$ is $m_p=2\cdot 3 + 3=9$.

\noindent Blow up $\P^5$ in $p$ and let $\P^4[\mathbf{l}]$ be the exceptional divisor. Then the strict transform $Y$ of $\overline{Y}$ admits exceptional locus given by quartic complete intersection surface $E:=\{f_2(\mathbf{l})=g_2(\mathbf{l})=0\}\subset\P^4[\mathbf{l}]$, which is smooth for $f_2,g_2$ sufficiently general (compare with \ref{dmu4}).

\noindent On the other hand the generic complete intersection $\widetilde{Y}$ of two cubic hypersurfaces in $\P^5$ is a \cy 3--fold. Then
\begin{itemize}
  \item $T(Y,\overline{Y},\widetilde{Y})$ \emph{is generically a type II g.t. whose exceptional divisor $E$ is a smooth del Pezzo surface of degree 4}.
\end{itemize}
Since $\widetilde{Y}$ is a projective complete intersection, its Betti numbers are well known and Theorem \ref{cambio omologico II +} gives $c'=5$, $c''=9$ and the following table
\begin{equation*}
\begin{tabular}{lclclclc}
  \hline \hline
    \hskip11.4pt   \vrule     &\quad     &$b_2$ &\quad  & $b_3$&\quad
  &$b_4$\\
  \hline
  \hskip11.4pt\vrule&&&&&&&\\
  $Y$ \vrule   &\quad    & 2 &\quad & 134 &\quad & 2  &\\
  $\overline{Y}$ \vrule            &\quad      &  1 &\quad & 139&\quad  & 1 &\\
  $\widetilde{Y}$ \vrule &\quad  & 1 &\quad & 148 &\quad & 1 &\\
   \hline\hline
\end{tabular}
\end{equation*}

\subsection{The elliptic case}\label{caso ellittico}
Let us now assume that the type II g.t. (\ref{typeII}) admits exceptional locus $E=\Exc(\phi)$ given by a normal and \emph{elliptic} del Pezzo surface (recall \ref{raz_o_ell} and \ref{ellittica}). Then Theorem \ref{cambio omologico tipo II} can be rewritten as follows:

\begin{theorem}\label{cambio omologico tipo II-ellittico} Assume that the type II g.t. (\ref{typeII}) admits a \emph{normal and elliptic} del Pezzo surface as exceptional divisor $E=\Exc(\phi)$. Then it admits homological type $h[T]=(0,1,c',c'')$ as in Theorem \ref{cambio omologico tipo II} and moreover
\begin{itemize}
    \item[(i)] $b_2(E)=1$ and $b_3(E)=2$, giving $c'=-2$\ ,
    \item[(ii)] $c''=m_p$ and in particular it must be even.
\end{itemize}
\end{theorem}

The previous (ii) gives immediately the following

\begin{corollary}\label{Milnor-pari}
If $\phi:Y\rightarrow\overline{Y}$ is a primitive type II smooth resolution such that $E=Exc(\phi)$ admits an elliptic singular point and $\overline{Y}$ is smoothable then the Milnor number $m_p$ of $p=\phi(E)$ is even.
\end{corollary}

\begin{example}\label{esempio-ell} The present example is aimed to give an account of how the Milnor number of $p=\phi(E)$ may jump when specializing to an elliptic exceptional del Pezzo surface $E$, respecting Corollary \ref{Milnor-pari}. Let us, in fact, consider the case $d=2$ in Examples \ref{esempi-tII}, whose general case gives Milnor number $m_p=27$. Specialize to consider
\begin{eqnarray*}
  f_4 &:=& y^2+g_2(x_1,x_2)\ y+g_4(x_1,x_2) \\
  f_5 &=& ax_3^5 + g_5(x_1,x_2,y) \\
  f_6 &=& bx_3^6 + g_6(x_1,x_2,y)
\end{eqnarray*}
where $g_k$ is a generic (weighted, for $k=5,6$) homogeneous polynomial of degree $k$ and $(a,b)\neq (0,0)$: this last condition is necessary to guarantee that $$p:=[1:0:0:0:0]\in\overline{Y}\subset\P:=\P(1,1,1,1,2)$$ is an \emph{isolated} singular point. Its local equation in $\C^4$ is given by
\begin{eqnarray*}
  y^2+g_2(x_1,x_2)\ y+g_4(x_1,x_2) + ax_3^5 &=& 0\quad \text{if $a\neq 0$} \\
  y^2+g_2(x_1,x_2)\ y+g_4(x_1,x_2) + bx_3^6 &=& 0\quad \text{if $a= 0$ and $b\neq 0$}\ .
\end{eqnarray*}
Milnor--Orlik Theorem (\cite{Milnor-Orlik} Theorem 1) then gives
\begin{eqnarray*}
  m_p &=& 3^2\cdot 4 = 36 \quad \text{if $a\neq 0$} \\
  m_p &=& 3^2\cdot 5 = 45\quad \text{if $a=0$ and $b\neq 0$}\ .
\end{eqnarray*}
The strict transform $Y$ of $\overline{Y}$, under the weighted blow up of $\P$ in $p$, admits as exceptional locus the \emph{elliptic del Pezzo} surface
\[
    E:=\{m^2+g_2(l_1,l_2)\ m+g_4(l_1,l_2)=0 \}\subset\P(1,1,1,2)[l_1,l_2,l_3,m]
\]
admitting an elliptic singular point in $q:=[0:0:1:0]\in\P(1,1,1,2)$. By Corollary \ref{Milnor-pari} we see that \emph{$Y$ can be a smooth resolution of $\overline{Y}$ only if $a\neq 0$}, since only in this case $m_p$ is even.
In fact, dividing by $l_3$, the equations of $\overline{Y}$ in the affine open subset
\[
    \mathcal{A}_{0,3}:=\{([x_0:\cdots :x_3:y],[l_1,l_2,l_3,m])\in\P(1,1,1,1,2)\times\P(1,1,1,2)|(x_0,l_3)\neq \mathbf{0}\}
\]
are given by $\mathbf{h}=(h_1,\ldots,h_4)=\mathbf{0}$ where
\begin{eqnarray*}
  h_1 &:=& x_1 - l_1x_3 \\
  h_2 &:=& x_2- l_2 x_3 \\
  h_3 &:=& y- m x_3^2 \\
  h_4 &:=& m^2+g_2(l_1,l_2)m+g_4(l_1,l_2) +ax_3 + x_3g_5(l_1,l_2,m) + bx_3^2 + x_3^2g_6(l_1,l_2,m)
\end{eqnarray*}
whose jacobian in $\mathbf{0}\in\mathcal{A}_{0,3}\cong\C^7$ gives
\[
    J_{\mathbf{0}}(\mathbf{h}):=\frac{\partial\mathbf{h}(\mathbf{0})}{\partial(x_1,\ldots,x_3,y,l_1,l_2,m)}=\left(
                                                                                  \begin{array}{ccccccc}
                                                                                    1 & 0 & 0 & 0 & 0 & 0 & 0 \\
                                                                                    0 & 1 & 0 & 0 & 0 & 0 & 0 \\
                                                                                    0 & 0 & 0 & 1 & 0 & 0 & 0 \\
                                                                                    0 & 0 & a & 0 & 0 & 0 & 0 \\
                                                                                  \end{array}
                                                                                \right)\ .
\]
Clearly $\rk(J_{\mathbf{0}}(\mathbf{h}))=4$ if and only if $a\neq 0$, as expected by Corollary \ref{Milnor-pari}.
Observe that if $a\neq 0$ then $T(Y,\overline{Y},\widetilde{Y})$ is a g.t. whose Betti number can be obtained by Theorem \ref{cambio omologico tipo II-nonnormale} giving $c'=-2$, $c''=36$ and
\begin{equation*}
\begin{tabular}{lclclclc}
  \hline \hline
    \hskip11.4pt   \vrule     &\quad     &$b_2$ &\quad  & $b_3$&\quad
  &$b_4$\\
  \hline
  \hskip11.4pt\vrule&&&&&&&\\
  $Y$ \vrule   &\quad    & 2 &\quad & 174 &\quad & 2  &\\
  $\overline{Y}$ \vrule            &\quad      &  1 &\quad & 172&\quad  & 1 &\\
  $\widetilde{Y}$ \vrule &\quad  & 1 &\quad & 208 &\quad & 1 &\\
   \hline\hline
\end{tabular}
\end{equation*}
Compare with table (\ref{tabella_d=2}).
\end{example}

\subsubsection{Local analysis of the elliptic singularity} The proof of Theorem \ref{cambio omologico tipo II-ellittico} come from the study of a minimal resolution, of a complex smoothing and of the Milnor number of the elliptic singular point of $E$.

\begin{lemma}[The resolution]\label{resolution_d<4}
Let $\pi:\widehat{E}\rightarrow E$ be a minimal resolution of the elliptic del Pezzo surface $E$. Then
\begin{equation}\label{omologia risolta}
    b_0(\widehat{E})=b_4(\widehat{E})=1\ ,\ b_1(\widehat{E})=b_3(\widehat{E})=2\ ,\ b_2(\widehat{E})=h^{1,1}(\widehat{E})= 2 \  .
\end{equation}
\end{lemma}

\begin{proof} First equations are clearly obvious. By \ref{ellittica}, $\widehat{E}\cong\P(\co_C\oplus\mathcal{L})$ is a geometrically ruled surface over an elliptic curve $C$. Then
\[
    \forall i>0\quad h^i(\co_{\widehat{E}})=h^i(\co_C)\quad\text{(\cite{Hartshorne} Lemma V.2.4)}
\]
giving $h^1(\co_{\widehat{E}})=1$ and $h^2(\co_{\widehat{E}})=0$. These relations and Kodaira--Serre duality suffice to end up the proof.
\end{proof}

\begin{lemma}[The smoothing]\label{smoothing_d<4}
Let $\widetilde{E}$ be a smoothing of the elliptic del Pezzo surface $E$. Then
\begin{equation}\label{omologia lisciata}
    b_0(\widetilde{E})=b_4(\widetilde{E})=1\ ,\ b_1(\widetilde{E})=b_3(\widetilde{E})=0\ ,\ b_2(\widetilde{E})= 10-d
\end{equation}
where $d=\deg(E)$.
\end{lemma}

\begin{proof}
Recall that $1\leq d\leq 3$, by \ref{dmu4}. Moreover:
\begin{itemize}
    \item if $d=3$ then $E$ is described in \ref{d=3} and $\widetilde{E}$ is a generic cubic surface in $\P^3$,
    \item if $d=2$ then $E$ is described in \ref{d=2} and $\widetilde{E}$ is a generic quartic surface in $\P(1,1,1,2)$,
    \item if $d=1$ then $E$ is described in \ref{d=1} and $\widetilde{E}$ is a generic degree 6 surface in $\P(1,1,2,3)$.
\end{itemize}
Therefore $\widetilde{E}$ turns out to be isomorphic to a smooth degree $d$ del Pezzo surface. Then (\ref{omologia lisciata}) follow by (\ref{omologia eccezionale-sing}) with $k_E=c_E=0$.
\end{proof}

\begin{lemma}[The Milnor number]\label{Milnor-ell} Let $q$ be the unique (elliptic) singular point of an elliptic del Pezzo surface $E$ of degree $\deg(E)=d$.
Then the Milnor number of $q\in E$ is given by $m_q=11-d$.
\end{lemma}

\begin{proof} By \ref{d=3}, if $d=3$ then $q$ is the vertex of a cone over a cubic plane curve $f_3(x_1,x_2,x_3)=0$. Then $m_q=2^3=8=11-3$.

\noindent By \ref{d=2}, if $d=2$ then $E$ is a double covering of $\P^2$ ramified along four distinct lines meeting in the elliptic singular point $q\in E$, whose local equation in $\C^2$ is then given the degree 4 homogeneous polynomial in two variables giving the four lines. Then $m_q=3^2=9=11-2$.

\noindent By \ref{d=1}, if $d=1$ then $E$ is the double covering of $\P(1,1,2)[x_0,x_1,y]$ ramified along the reducible degree 6 curve $\{\prod_{i=1}^3(y-q_i(x_1)=0\}$, where $\deg q_i=2$. The elliptic singular point $q$ is then given by $[1:0:0]\in\P(1,1,2)$, whose Milnor number is $m_q=2\cdot 5 =10=11-1$, by the Milnor--Orlik Theorem (\cite{Milnor-Orlik} Theorem 1).
\end{proof}

\subsubsection{Proof of Theorem \ref{cambio omologico tipo II-ellittico}}  Let us recall that $\pi:\widehat{E}\rightarrow E$ is the contraction of the minimal section of $\widehat{E}\cong\P(\co_C\oplus\mathcal{L})$ (see \ref{ellittica}). Then setting $\{q\}=Q:=\Sing(E)$ and $L_q:=\pi^{-1}(q)$ one has
\begin{equation}\label{sezione}
    L=\Exc(\pi)=L_q\cong C
\end{equation}
and the total number of irreducible components of $L$ is given by $n_E=n_q=1=b_2(L)$ (notation as in the Proof of theorem \ref{cambio omologico tipo II}). In the present case, (\ref{rel-exc}) has to be rewritten  as follows
\[
    b_3(\widehat{E})+b_2(L)-b_2(\widehat{E})-b_1(L)+b_1(\widehat{E})+b_0(L)=b_3(E)-b_2(E)+b_1(E)+|Q|
\]
which, by (\ref{omologia risolta}) and (\ref{sezione}), gives
\begin{equation}\label{rel-ell}
    b_3(E)-b_2(E)+b_1(E)=1\ .
\end{equation}
On the other hand, calling $\widetilde{A}$ the Milnor fiber near $q\in E$, it has the homotopy type of a bouquet of 2--spheres, by Theorem \ref{i.s.omotopia} and the fact that $q\in E$ is an isolated hypersurface singularity. By Lefschetz Duality, Ehresmann diffeomorphism and relative homology of the couple $(E,Q)$, one has
\[
    \forall i\geq 2\quad H_i(\widetilde{E},\widetilde{A})\cong H^i(\widetilde{E}\setminus\widetilde{A})\cong H^i(E\setminus Q)\cong H_i(E,Q)\cong H_i(E)\ .
\]
Then the relative homology long exact sequence of the couple $(\widetilde{E},\widetilde{A})$ gives the following exact sequence
\begin{equation*}
    \xymatrix{0\ar[r]&H_3(\widetilde{E})\ar[r]&H_3(E)\ar[r]&H_2(\widetilde{A})\ar[r]&H_2(\widetilde{E})\ar[r]&H_2(E)\ar[r]&0}\ .
\end{equation*}
Apply relations (\ref{omologia lisciata}) and Lemma \ref{Milnor-ell} to such a sequence to get
\begin{equation*}
    b_3(E)-b_2(E)=m_q - 10 +d= 1\ .
\end{equation*}
Recall (\ref{rel-ell}) to get then:
\begin{equation}\label{rel-ellittiche}
    b_0(E)=b_4(E)=1\quad,\quad b_1(E)=0\quad,\quad b_3(E)-b_2(E)=1\ .
\end{equation}
Then point (i) in the statement follows by (\ref{rel-ellittiche}) and observing that the birational contraction $\pi:\widehat{E}\rightarrow E$ is obtained as the contraction of
    the minimal section of $\widehat{E}$ (\cite{Hidaka-Watanabe81},
    Theorem 2.2), whose class is a generator of $H^2(\widehat{E},\Q)\cong\Q^2$.

    \noindent Let us now observe that, by replacing (\ref{omologia eccezionale-sing}) with relations (\ref{rel-ellittiche}), it is still possible to write down equations (\ref{relazioni}) in the present case, proving equations in the first line in Definition \ref{ht}.(a). On the other hand, arguments proving (\ref{Pic-invariante}) and (\ref{Pic-sing}) still hold in the present case, proving the second and the fourth line in Definition \ref{ht}.(a) with $k'=0$ and $k=k''=1$.
Since $\deg E\leq 3$ we can apply relations (\ref{omologia milnor}) to equations (\ref{b3}), to get immediately (ii) in the statement and
\begin{eqnarray}\label{b3-ell}
  b_3(Y) &=& b_3(\overline{Y}) +1 - b_2(E)+b_3(E) = b_3(\overline{Y}) +2  \\
  \nonumber
  b_3(\overline{Y}) &=& b_3(\widetilde{Y}) + \chi(\widetilde{B})-1= b_3(\widetilde{Y}) - m_p\ ,
\end{eqnarray}
proving the third line in (a), and consequently equations in (b), of Definition \ref{ht}, with $c',c''$ as in (i) and (ii), respectively.
\hfill$\Box$

\subsection{The non-normal case}\label{caso non-normale} Let us now assume that the type II g.t. (\ref{typeII}) admits exceptional locus $E=\Exc(\phi)$ given by a \emph{non-normal} del Pezzo surface (recall \ref{non-normale}, cases (i) and (ii)). Then Theorem \ref{cambio omologico tipo II} can be rewritten as follows:

\begin{theorem}\label{cambio omologico tipo II-nonnormale} Assume that the type II g.t. (\ref{typeII}) admits a \emph{non-normal} del Pezzo surface as exceptional divisor $E=\Exc(\phi)$. Then it admits homological type $h[T]=(0,1,c',c'')$ as in Theorem \ref{cambio omologico tipo II} and moreover
\begin{itemize}
    \item[(i)] \emph{$b_2(E)\in\{1,2\}$ and $b_3(E)=0$, giving $c'\in\{0,1\}$, respectively},
    \item[(ii)] $\chi(\widetilde{B})\equiv b_2(E)\mod 2 $\ .
\end{itemize}
\end{theorem}

\begin{corollary} If ($\ref{typeII}$) is a type II g.t. with $E=\Exc(\phi)$ a non-normal del Pezzo surface, then $E$ is like in \ref{non-normale}.(i) (resp. \ref{non-normale}.(ii)) if and only if $\chi(\widetilde{B})$ is even (resp. odd).
\end{corollary}

\begin{proof}[Proof of Theorem \ref{cambio omologico tipo II-nonnormale}] Let us first of all observe that the normalization $\F_a=\P(\co\oplus\co_{\P^1}(-a))$ of $E$ has the following Betti numbers:
\begin{equation}\label{Betti_Fa}
    b_0(\F_a)=b_4(\F_a)=1\quad,\quad b_1(\F_a)=b_3(\F_a)=0\quad,\quad b_2(\F_a)=h^{1,1}(\F_a)=2\ .
\end{equation}
In fact, the first equalities are obvious. Moreover,
\[
    \forall i>0\quad h^i(\co_{\F_a})=h^i(\co_{\P^1})=0
\]
giving the middle equalities in (\ref{Betti_Fa}) and $b_2(\F_a)=h^{1,1}(F_a)$. In particular, $\chi(\co_{\F_a})=1$ and the Noether formula (\cite{BPvdV84} Theorem I.(5.4)) gives
\begin{equation}\label{Noether_Fa}
    2+h^{1,1}(\F_a)=\chi(\F_a)=12 \chi(\co_{\F_a})-K_{\F_a}^2=4\ \Longrightarrow\ h^{1,1}(\F_a)=2\ ,
\end{equation}
where $K_{\F_a}=-2C_0-(a+2)f$ is the canonical divisor of $\F_a$, giving $K_{\F_a}^2=8$ (notation as in \ref{non-normale}; see \cite{Hartshorne} Corollary V.2.11). Then (i) follows immediately by (\ref{Betti_Fa}) recalling the construction of the normalization morphism $\F_a\rightarrow E$ described in \ref{non-normale}.(i) and (ii). Moreover, we get
\begin{equation}\label{omologia eccezionale-sing-nn}
    b_0(E)=b_4(E)=1\quad,\quad b_1(E)=b_3(E)=0\quad,\quad b_2(E)\in\{1,2\}\ ,
\end{equation}
which replaced to (\ref{omologia eccezionale-sing}) guarantee that (\ref{relazioni}), (\ref{Pic-invariante}) and (\ref{Pic-sing}) still hold, by the same arguments. Then equations in Definition \ref{ht}.(c), with $k,c',c''$ as in Theorem \ref{cambio omologico tipo II}, give
\begin{eqnarray*}
  \chi(Y)-\chi(\overline{Y}) &=& 2-b_3(Y)+b_3(\overline{Y})=1+b_2(E)=\chi(E)-1 \\
  \chi(\overline{Y}) - \chi(\widetilde{Y})&=& -b_3(\overline{Y})+b_3(\widetilde{Y})=1-\chi(\widetilde{B})
\end{eqnarray*}
which is enough to prove (a) and (b) in Definition \ref{ht} with $k'=0,k=k''=1$ and $c',c''$ assigned by (i) and (ii).
\end{proof}

\end{document}